\UseRawInputEncoding
\documentclass[hidelinks,12pt]{amsart}
\usepackage{amscd,amssymb,graphics,color,a4wide,hyperref}
\usepackage{graphicx}
\usepackage{psfrag}
\usepackage{pb-diagram}
\usepackage{verbatim}
\usepackage{epsf}

\usepackage{amsmath,amsthm,amssymb,enumerate}
\usepackage{latexsym}
\usepackage{amscd}
\usepackage{amssymb}
\usepackage{mathrsfs}
\usepackage{eufrak}
\usepackage[all,arc,curve,color,frame]{xy}
\usepackage{MnSymbol}
\usepackage{float}
\usepackage{color}

\usepackage{mathrsfs}
\input xy
\xyoption{all}

\DeclareFontFamily{U}{rsf}{}
\DeclareFontShape{U}{rsf}{m}{n}{
  <5> <6> rsfs5 <7> <8> <9> rsfs7 <10->  rsfs10}{}
\DeclareMathAlphabet{\mathscr}{U}{rsf}{m}{n}

\newtheorem{theorem}{Theorem}[section]
\newtheorem{lem}[theorem]{Lemma}
\newtheorem{prop}[theorem]{Proposition}

\theoremstyle{definition}
\newtheorem{defn}[theorem]{Definition}

\newtheorem{ex}[theorem]{Example}

\newtheorem{claim}[theorem]{Claim}

\theoremstyle{remark}
\newtheorem{rem}[theorem]{Remark}

\numberwithin{equation}{section}

\newcommand{\mb}{\mathbb}
\newcommand{\s}{\sigma}
\newcommand{\ms}{\mathscr}
\newcommand{\mc}{\mathcal}
\newcommand{\mm}{\mathrm}

\newcommand{\nn}{\nonumber}
\newcommand{\tb}{\textbf}
\newcommand{\St}{\tilde{\mb{S}}[t]}

\newcommand\blfootnote[1]{%
  \begingroup
  \renewcommand\thefootnote{}\footnote{#1}%
  \addtocounter{footnote}{-1}%
  \endgroup
}

\def\mydate{\ifcase\month \or January\or February\or March\or
April\or May\or June\or July\or August\or September\or October\or
November\or December\fi \space\number\day,\space\number\year}

\newlength{\picwidth} 
\newlength{\miniwidth} 
\newlength{\nanowidth} 
\newlength{\melowidth} 
\newlength{\leftminiwidth} 
\newlength{\rightminiwidth} 
\newlength{\minipagewidth} 

\begin{document}
\def\mapright#1{\smash{
 \mathop{\longrightarrow}\limits^{#1}}}
\def\mapleft#1{\smash{
 \mathop{\longleftarrow}\limits^{#1}}}
\def\exact#1#2#3{0\to#1\to#2\to#3\to0}
\def\mapup#1{\Big\uparrow
  \rlap{$\vcenter{\hbox{$\scriptstyle#1$}}$}}
\def\mapdown#1{\Big\downarrow
  \rlap{$\vcenter{\hbox{$\scriptstyle#1$}}$}}
\def\dual#1{{#1}^{\scriptscriptstyle \vee}}
\def\invlim{\mathop{\rm lim}\limits_{\longleftarrow}}
\def\rto{\raise.5ex\hbox{$\scriptscriptstyle ---\!\!\!>$}}

\input epsf.tex
\title
[The quantum $A_{\infty}$-relations on the elliptic curve]
{The quantum $A_{\infty}$-relations on the elliptic curve}
\author{Michael Slawinski}
\address{UCSD Mathematics, 9500 Gilman Drive, La Jolla, CA 92093-0112, USA}
\email{slawinski.michael@gmail.com}


\maketitle
\tableofcontents
\bigskip

\blfootnote{Warm thanks to Mark Gross for conceiving this project and for his prudent guidance throughout.  This research was partly supported by NSF grant DMS-0805328.}


\begin{abstract}
 We define and prove the existence of the Quantum $A_{\infty}$-relations on the Fukaya category of the elliptic curve, using the notion of the Feynman transform of a modular operad, as defined by Getzler and Kapranov.  Following Barannikov, these relations may be viewed as defining a solution to the quantum master equation of Batalin-Vilkovisky geometry.
\end{abstract}

\section*{Introduction}

Getzler and Kapranov introduced the notion of a modular operad in \cite{Mod}.  Following Barannikov's work in \cite{Bar}, we endow the Fukaya category of the elliptic curve with the structure of an algebra over the Feynman transform of a twisted modular operad.  Leveraging this structure, we present an explicit formulation of what is called the \emph{Quantum $A_{\infty}$-relations} on the elliptic curve.  This set of relations is a generalization of the well-known $A_{\infty}$-relations and defines the solutions to the quantum master equation of Batalin-Vilkovisky geometry on affine $P$-manifolds, where a $P$-manifold is a manifold which is odd symplectic.  The categories on which these modular operads are defined are categories of stable graphs, indexed by the number of legs and the dimension of the first homology group of the graphs in the given category.  The version of the Fukaya category used in this work is given combinatorially by a simple generalization of tropical Morse trees (chapter 8 of \cite{Clay}), called tropical Morse graphs, as defined in \cite{S}.

Barannikov, in \cite{Bar}, uses the Feynman transform of Getzler and Kapranov \cite{Mod} to claim the existence of a combinatorial description of Gromov-Witten invariants, in terms of the periodic cyclic homology of the given twisted modular operads via the characteristic class map.  He further suggests the possibility of constructing a Feynman transform-algebra structure on the Lagrangians of a symplectic manifold.  

The result of this paper is the construction of an explicit example of such a structure, i.e., the construction of a morphism from the Feynman transform of a twisted modular operad, generated by the elements of the symmetric group, to a modular operad generated by the Lagrangian submanifolds of an elliptic curve.  Using the chain complex nature of this morphism, we explicitly construct a set of relations between compositions (contractions via a bilinear form) of tensors in the Fukaya category.

These tensors are defined by summing over countable sets of tropical Morse trees and graphs (holomorphic disks and annuli, resp.), and the proof of the existence of these relations is done by examining degenerations of one-dimensional moduli spaces of these trees and graphs on the elliptic curve.  The degeneration of a moduli space of trees corresponds to a disk bubbling off from another disk and the degeneration of a space of graphs corresponds either to a disk bubbling off from an annulus, or an annulus pinched into a disk.  In terms of the modular operad, each composition is determined by the contraction of a certain edge of a stable graph, the fatgraph of which is homeomorphic to each point in the interior of the degenerating moduli space.

The relations are graded by numbers $b=\dim\mm{H}_1(G)+\sum_{\mm{Vert}(G)} b(v)$, where the $b(v)$ are integers which control the way in which the flags attached to the vertices $v$ of a stable graph $G$ are organized into cycles.  In general, these cycles provide the combinatorial data necessary to build the corresponding Riemann surface (fatgraph) from the given stable graph.  In this work, we have $b(v)=0$ for all vertices $v$ in a given stable graph, and each Riemann surface is built using either  a tropical Morse tree or graph, depending on whether $b=\dim\mm{H}_1(G)$ is equal to 0, or 1, respectively.  The $b=0$ relations are the usual $A_{\infty}$-relations, and the $b=1$ relations are the new Quantum $A_{\infty}$-relations.  The latter may be viewed as a sort of linear dependence, as the relation is just the equality between a countable sum of signed tensors and the zero vector.  It must be noted that the $A_{\infty}$ structure established in this paper is a pre-algebra structure, and not a complete algebra structure.

The tensors which satisfy these relations are precisely the solutions to the quantum master equation described by Barannikov in \cite{Bar}.  Using the solutions to the Quantum master equation, Barannikov constructs a partition function (see \cite{Kontsevich0}) that is used to define classes in the compactification $\overline{\mc{M}}_{n,g}$.  These classes, in turn, may provide a combinatorial description of Gromov-Witten invariants on the given manifold.  One may view this route to the construction of positive genus Gromov-Witten invariants on the elliptic curve as being parallel to the one described by C\v{a}ld\v{a}raru and Tu in \cite{CaTu}. 

All chain complexes in this paper are cohomological.  The departure from the homological convention adopted by
Getzler and Kapranov in \cite{Mod} was necessary for the sake of consistency with the gradings of the Fukaya Category.  We adopt the supercommutativity conventions $x\otimes y=(-1)^{|x||y|}y\otimes x$ and
$(f\otimes g)(x\otimes y) = (-1)^{|g||x|}f(x)\otimes g(y)$, for vectors $x,y$ and maps of vectors $f,g$, where $|\,\bullet\,|:=\deg(\,\bullet\,)$.

\subsection{Acknowledgements}
Warm thanks to Mark Gross for his prudent guidance throughout this project.  This research was partly supported by NSF grant DMS-0805328.




\section{Modular Operads}
The material in this first section is expository and may be found in \cite{Mod} or \cite{Bar}.  Some definitions have been altered slightly so as to better mesh with the remaining material.

\begin{defn}
A graph $G$ is a triple $(\text{Flag}(G),\lambda,\sigma)$, where $\text{Flag}(G)$
is a finite set, whose elements are called flags, $\lambda$ is a partition of $\text{Flag}(G)$, and
$\sigma$ is an involution acting on $\text{Flag}(G)$.  The vertices of $G$ are the unordered blocks (subsets) of the partition and the set of all such is denoted by $\text{Vert}(G)$. We write
$\text{Leg}(v)$ for the subset of $\text{Flag}(G)$ corresponding to the vertex $v$, and define the \emph{valence} of $v$ as the cardinality of $\text{Leg}(v)$.  Let
$\text{Edge}(G)$ denote the set of two-cycles of $\sigma$ and let $\text{Leg}(G)$ denote the subset
of $\text{Flag}(G)$ fixed by $\sigma$. We say that two legs $meet$ if they either belong to the
same vertex, or comprise an edge.  A leg is \emph{external} if it remains fixed under $\sigma$. 
\end{defn}

\begin{ex}

$\text{Flag}(G) = \{1,\dots,9\}$, $\lambda = \{1,2,3\}\cup\{4,5,6\}\cup\{7,8,9\}$,  $\sigma =
(12)(34)(57)(68)$

\begin{figure}[h]
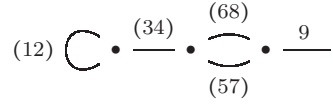

\[  \xygraph{
!{<0cm,0cm>;<1cm,0cm>:<0cm,1cm>::} !{(0,0) }*+{_{\bullet}}="a" !{(1,0) }*+{_{\bullet}}="b" !{(2,0)
}*+{_{\bullet}}="c" !{(3,0) }*+{}="d" "a"-"b"^{(34)} "c"-"d"^{9} "b"-@/_/"c"_{(57)} "b"-@/^/"c"^{(68)}
"a" -@(lu,ld) "a"_{(12)} }  \]
\caption{$G$ has three vertices, four edges, and one external leg.}
\label{firstgraph}
\end{figure}
\end{ex}

\begin{defn}\label{stablegraph}
A \emph{stable} graph $G$ is a connected graph with a non-negative integer $b(v)$ assigned to each vertex $v\in\mm{Vert}(G)$, such that $2b(v)+n(v)-2>0$ for any $v\in\mm{Vert}(G)$, where $n(v)$ is the valence of $v$.  We note that since the complexes used in this work are cohomological, the stability condition is more appropriately written $n(v)>2-2b(v)$.  For a stable graph $G$ we put $b(G)=\sum_{v\in\mm{Vert}(G)}b(v)+\dim \mm{H}_1(G)$, where $\mm{H}_1(G)$ is computed by replacing $G$ with the corresponding
CW-complex in the obvious way (See Definition \ref{topologicalG}). 
\end{defn}

If $G$ is contractible, then $G$ will be called a
$tree$.  We can break symmetry and call one of the legs the $output$, and the rest of the legs will be called
the $inputs$.  A tree of this type is a $rooted$ $tree$.

\begin{prop}\label{EdgeVsVertices}
If $G$ is a stable graph, then $$|\mm{Edge}(G)|=b(G)-1+\sum_{\mm{Vert}(G)}(1-b(v)).$$
Note that when $b(v)=0$ for all $v\in\mm{Vert}(G)$ this formula reduces to Euler's formula for planar graphs, i.e., $|\mm{Edge}(G)|=\dim\mm{H}_1(G)-1+|\mm{Vert}(G)|$, where the number of ``faces" of a planar graph $G$ is given by $\dim\mm{H}_1(G)+1$.
\end{prop}
\begin{proof}
See \cite{Mod}.
\end{proof}

Let $G$ and $G'$ be two graphs.  Let $S\subseteq\mm{Leg}(G)$ and $S'\subseteq\mm{Leg}(G')$ be two subsets such that there exists a bijection $\phi:S\rightarrow S'$.  We can form a new graph $G\coprod_{\varphi}G'$
by gluing the elements of $S$ to the elements of $S'$ using $\varphi$.  Then
$$\mathrm{Edge}(G\coprod_{\varphi}G^{'})=\mathrm{Edge}(G)\cup \mathrm{Edge}(G^{'})\cup \{(i\,\varphi(i))\}$$ for $i\in S$
and $$\mathrm{Leg}(G\coprod_{\varphi}G')\subseteq \mathrm{Leg}(G)\sqcup \mathrm{Leg}(G^{'}).$$  We
can then decompose any graph $G$ as
$$
G = \coprod_{v\in \mathrm{Vert}(G)}*_{n(v),b(v)}^l,
$$
where $*_{n(v),b(v)}^l$ is the unique graph with one vertex, $l$ loops, $n(v)=|\mm{Leg}(v)|-2l$ legs, and $v$ is labeled by $b(v)$.  I write $*_{n,b}$ for $*_{n,b}^0$.  The bijection $\varphi$ is given implicitly by $G$.

\begin{defn}
Let $f:G_0\rightarrow G_1$ be a morphism as defined in section 2.13 of \cite{Mod}, and let $v\in \mathrm{Vert}(G_1)$.  Define $f^{-1}(v)\subseteq\mm{Flag}(G_0)$ as the set of all flags which either comprise
edges collapsed to $v$ by $f$, or are the flags attached to the vertices which bound the collapsed edges.
The corresponding graph, which will also be denoted $f^{-1}(v)$, is uniquely defined by the arrangement of these flags in $G_0$.
\end{defn}

\begin{defn}
Denote by $\Gamma((n,b))$ the category whose objects are pairs $(G,\rho)$
where $G$ is a stable graph and $\rho$ is a bijection between $\mm{Leg}(G)$ and the set
$\{1,\dots, n\}$, and whose morphisms are morphisms of stable graphs preserving the
labeling $\rho$ of the legs.  The graph $*_{n,b}$ defined above is the final object of this category.
\end{defn}

For any subset $I\subseteq \mm{Edge}(G)$, let $G/I$ denote the graph obtained by contracting each element of $I$.  Let $\left[\Gamma((n,b))\right]$ denote the set of isomorphism classes of this category, which is finite by Lemma 2.16 of \cite{Mod}.

\begin{defn}(Topological Definition for $\Gamma((n,b))$):\label{topologicalG}
For $G$ a stable graph, let $|G|$ denote the corresponding 1-dimensional CW-complex.  Replace each object $G$ of $\Gamma((n,b))$ with $|G|$, and each morphism $f:G_0\longrightarrow G_1$ with a continuous map $\varphi_f:|G_0|\longrightarrow |G_1|$ which is constant on edges which collapse, and is the identity everywhere else.  The value on each collapsed edge is given by $f|_{\mm{Edge}(f^{-1}(v))}=v$, as every collapsed edge belongs to a set of the form $\mm{Edge}(f^{-1}(v))$, for some $v\in\mm{Vert}(G_1)$.
\end{defn}

\begin{defn}\label{StableSModDef}
A stable $\mb{S}$-module is a collection of chain complexes $\{P((n,b))\}$ with an action of $\mb{S}_n$ on $P((n,b))$, and such that $P((n,b))=0$ if $2b+n-2 \leq0$.  A morphism of stable $\mb{S}$-modules $P\rightarrow Q$ is a collection of equivariant maps of chain complexes $P((n,b))\rightarrow Q((n,b))$.  There is a natural extension of this definition over a finite set $I$.  See Section 2.1 of \cite{Mod} for further details.
\end{defn}

\begin{defn}\label{monad}
Let $P$ be a stable $\mb{S}$-module and let $G$ be a stable graph.  A stable $\mb{S}$-module $P((n,b))$ has a natural extension over a finite set, such as $Leg(v)$.  Set $P((G))=\bigotimes_{\mm{Vert}(G)}P((\mm{Leg}(v),b(v)))$.  One may then write
\begin{equation}\label{tensor}
    P((G\coprod_{\varphi}G^{'})) = P((G))\otimes P((G^{'}))
\end{equation}
for any $G, G^{'}, \varphi$.

\end{defn}

\begin{defn}
Let $V$ be a chain complex over a field of characteristic zero and equip $V$ with a group action $A\times V\longrightarrow V$ for $A$ a finite group.  Denote by $V_A$ the chain complex $V/<\{av-v|a\in A,v\in V\}>$ of \emph{co-invariants}.
\end{defn}

\begin{defn}\label{Pre-Op}
A \emph{modular pre-operad} is a stable $\mb{S}$-module $P$ together with a chain map
\begin{equation}\label{ModOp}
    \mu:\mb{M}P((n,b)):=\bigoplus_{G\in[\Gamma((n,b))]}P((G))_{\mm{Aut}(G)}\longrightarrow P((n,b))
\end{equation}
called the \emph{structure map}. 
The group $\mm{Aut}(G)$ is the set of homeomorphisms from $G$ to itself which fix the legs.  In particular, $\mm{Aut}(*_{n,b})$ is the trivial group.

See Section 2.17 of \cite{Cyclic} for a detailed discussion of the endofunctor $\mb{M}$ and the structure map $\mu:\mb{M}\mathcal{A}\longrightarrow\mathcal{A}$ for a stable $\mb{S}$-module $\mathcal{A}$.   

For the specific modular operads considered in this work, the maps $P((G))_{\mm{Aut}(G)}\longrightarrow P((n,b))$ result from the application of the appropriate power of some bilinear form $B$ to the chain complex $P((G))$. 

\end{defn}
If $P$ is a modular pre-operad and $G\in\mm{Ob}\,\Gamma((n,b))$, denote by
\begin{equation}\label{mu}
    \mu_G:P((G))\longrightarrow P((n,b))
\end{equation}
the $\mb{S}_n$-equivariant map obtained by composing the universal map (projection onto co-invariants followed by inclusion)
$$P((G))\longrightarrow \mb{M}P((n,b))=\bigoplus_{G\in[\Gamma((n,b))]}P((G))_{\mm{Aut}(G)}$$
with the structure map $\mu:\mb{M}P((n,b))\longrightarrow P((n,b))$.  This map is called composition along the graph $G$, and should be viewed as being parameterized by the contraction of the edges of $G$.  In this paper $\mu$ takes the form of a trace map.
\begin{defn}
Given a morphism $f:G_0\longrightarrow G_1$ of stable graphs, define the morphism $P((f)):P((G_0))\longrightarrow P((G_1))$ to be the composition
\begin{equation}\label{P(f)}
    P((G_0))=\bigotimes_{u\in\mm{Vert}(G_0)}P((\mm{Leg}(u),b(u)))\simeq\bigotimes_{v\in\mm{Vert}(G_1)}P((f^{-1}(v)))\qquad\qquad
\end{equation}
\begin{center}$
\begin{diagram}
\node{}\arrow{e,t}{\otimes_v\mu_{f^{-1}(v)}}\node{\bigotimes_{v\in\mm{Vert}(G_1)}P((\mm{Leg}(v),b(v)))=P((G_1))}
\end{diagram}$
\end{center}
\end{defn}

\begin{defn}\label{ModOpDef}
A \emph{modular operad} is a modular pre-operad $P$ such that
for any $G\in\Gamma((n,b))$ and $f\in\mm{Mor}(\Gamma((n,b)))$, the associations $$\begin{array}{ccc}
                                                                  G & \mapsto & P((G)) \\
                                                                  f & \mapsto & P((f))
                                                                \end{array}$$
define a functor from the category of stable graphs to the category of chain complexes over some field $k$.  In other words, $P((f\circ g))=P((f))\circ P((g))$ for any composition
\begin{center}
$\begin{diagram}
\node{G_0}\arrow{e,t}{g}\node{G_1}\arrow{e,t}{f}\node{G_2}
\end{diagram}$
\end{center}
of morphisms of stable graphs.
\end{defn}

For any morphism of graphs $G_0\longrightarrow G_1$, the morphism $P(G_0)\longrightarrow P(G_1)$ is a morphism of chain complexes.  So the equation
 $$\otimes_v\mu_{f^{-1}(v)}\circ \mm{d}_{P((G))}=\mm{d}_{P((G/J))}\circ\otimes_v\mu_{f^{-1}(v)}$$holds, where $\mm{d}_{P((G))}$ is the differential on $P((G))$, and $v$ runs through $\mm{Vert}(G_1)$.

\subsection{The Feynman Transform}
\subsubsection{The Free Modular Operad $\mb{M}P$}

Given a stable $\mb{S}$-module $P$, the stable $\mb{S}$-module underlying $\mb{M}P$ was defined in \eqref{ModOp} as $$\mb{M}P((n,b))=\bigoplus_{G\in[\Gamma((n,b))]}P((G))_{\mm{Aut}(G)}$$
If $P$ is a stable $\mb{S}$-module, then the composition map $\mu_G^{\mb{M}P}:\mb{M}P((G))\longrightarrow \mb{M}P((n,b))$ can be constructed by first mapping $\mb{M}P((G))$ to $P((G))$, projecting to coinvariants, and finally, including $P((G))_{\mm{Aut}(G)}$ into $\mb{M}P((n,b))$.

The map $\mb{M}P((G))\longrightarrow P((G))$ is constructed as follows.  Take $G\in\left[\Gamma((n,b))\right]$
and let $S_G = \left\{\oplus \eta_j \,|\,\eta_j:\{v_j\}\rightarrow\left[\Gamma((\mathrm{Leg}(v_j),b(v_j)))\right] \right\}$, where $\eta_j$
is a function that assigns to each vertex $v_j$ in $\mm{Vert}(G)$ a graph $G_{v_j}$ in $[\Gamma((\mm{Leg}(v_j),b(v_j)))]$.  Then
\begin{eqnarray}
  \mathbb{M}P((G)) &=& \bigotimes_{v\in\mm{Vert}(G)}\mathbb{M}P((\mm{Leg}(v), b(v))) \nn\\
   &=& \bigotimes_{v\in\mm{Vert}(G)}\left(\bigoplus_{H\in \left[\Gamma((\mathrm{Leg}(v),b(v)))\right]}P((H))_{\mathrm{Aut}(H)}\right) \nn\\
   &=& \bigoplus_{S_G}\left(\bigotimes_{\mathrm{Vert}(G)}P\left((\eta_j(v_j))\right)_{\mathrm{Aut}(\eta_j(v_j))}\right) \nn\\
   &=& \bigoplus_{S_G}\left(\bigotimes_{\mathrm{Vert}(G)}P((G_{v_j}))_{\mathrm{Aut}(G_{v_j})}\right) \nn
\end{eqnarray}
By taking $\eta_j=\mm{id}$ for all vertices $v_j$ in $\mathrm{Vert}(G)$ and noting that
$\mathrm{Aut}(*_{n,b})\simeq\mathrm{id}$ for any stable pair $(n,b)\in \mb{Z}_{\geq1}\times\mb{Z}_{\geq0}$, one
sees that $P((G))$ is indeed a summand of $\mb{M}P((n,b))$, and the map $\mb{M}P((G))\longrightarrow P((G))$ is the canonical projection.

Further insight into the structure of the composition $\mu_G^{\mb{M}P}$ may be obtained by considering a graph, which we denote by $G_{\eta}$, constructed from the graphs $G_{v}=\eta(v)$ above.  Let $G\in\Gamma((n,b))$ be such that $\mathrm{Vert}(G) = \{1, \dots,m\}$ with involution $\sigma$ and let $G_{v_i}\in\left[\Gamma((\mathrm{Leg}(v_i),b(v_i)))\right]$ for
$i = 1,\dots,m$.  For any category $\Gamma((n,b))$ there is a fixed labeling of the legs of any graph $G\in\Gamma((n,b))$, so every graph $G_{v_i}\in[\Gamma((\mm{Leg}(v_i),b(v_i)))]$ comes with a bijection $f_i:\mm{Leg}(G_{v_i})\longrightarrow\mm{Leg}(v_i) $.  It is important to recall that by a
$vertex$ we mean not only the associated point in the geometric realization, but the set of flags
emanating from that point as well.  Now, replace $v_i$ with $G_{v_i}$ and glue $l_{\alpha},l_{\beta} \in
\bigsqcup_i\mm{Leg}(G_{v_i})$ together if and only if $\sigma f_i(l_{\alpha}) =
f_j(l_{\beta})$ with $l_{\alpha}\in\mm{Leg}(G_{v_i})$ and $l_{\beta}\in\mm{Leg}(G_{v_j})$, that is, if and only if the legs $f(l_{\alpha})$ and $f(l_{\beta})$ form an edge in the original graph.  The resulting graph $G_{\eta}$ is such that $\mathrm{Leg}(G_\eta) = \mathrm{Leg}(G)$ and since $b(v_i) = b(G_{v_i})$, it must
be the case that $G_\eta$ belongs to $\Gamma((n,b))$.

The complex $\mb{M}P((G))$ should be viewed as a list of all possible complexes parameterized by graphs built from $G$ by replacing each vertex $v$ of $G$ with a new graph $G_v\in\Gamma((n(v),b(v)))$.  The restricted contraction map $\mu_G^{\mb{M}P}|_{\bigotimes_{\mathrm{Vert}(G)}P((G_{v_j}))_{\mathrm{Aut}(G_{v_j})}}$ is then the identification
of $\bigotimes_{\mathrm{Vert}(G)}P((G_{v_j}))_{\mathrm{Aut}(G_{v_j})}$ with $P((\tilde{G}))$, where $\tilde{G}$ is the `inflation' of
$G$, obtained by replacing each vertex $v_j$ of $G$ with $G_{v_j}$.



\begin{ex}

Let \[  \xygraph{
!{<0cm,0cm>;<2cm,0cm>:<0cm,2cm>::}
!{(5.2,0) }*+{_{\bullet}}="y"
!{(5.5,0) }*+{_{b(v)=0}}="kk"
!{(4.2,0) }*+{G}="z"
!{(4.6,0) }*+{=}="w"
!{(4.9,.4) }*+{_1}="aa"
!{(5.2,.5) }*+{_2}="bb"
!{(5.5,.4) }*+{_3}="cc"
!{(4.5,.5) }*+{}="dd"
"aa"-"y" "bb"-"y" "cc"-"y" "y" -@`{p+(-.5,-.5),p+(.5,-.5)} "y"_{(kk')}
}  \]
and let \[  \xygraph{
!{<0cm,0cm>;<2cm,0cm>:<0cm,2cm>::}
!{(5.2,0) }*+{_{\bullet}}="y"
!{(5.5,0) }*+{_{b(w)=0}}="kk"
!{(4.2,0) }*+{H}="z"
!{(4.6,0) }*+{=}="w"
!{(4.9,.4) }*+{_1}="aa"
!{(5.2,.5) }*+{_2}="bb"
!{(5.5,.4) }*+{_3}="cc"
!{(5.5,-.4) }*+{_{\ell'}}="a"
!{(4.9,-.4) }*+{_{\ell}}="b"
"aa"-"y" "bb"-"y" "cc"-"y" "y"-"a" "y"-"b"
}  \]
As $G$ has only one vertex $v$, each $\oplus\eta_j$ from above is replaced by a single map $\eta:\{v\}\longrightarrow [\Gamma((\mm{Leg}(v),b(v)))]$, and $S_G$ can be replaced by $[\Gamma((\mm{Leg}(v),b(v)))]$.  Let $G_{\eta_0}$ correspond to the $\eta$ such that $\eta(v)=H$.  The bijection $f$ in this case is $$f:\mm{Leg}(H)\longrightarrow \mm{Leg}(v)$$
$$1,2,3,\ell,\ell'\mapsto 1,2,3, k,k',$$
and applying the transposition $\s=(kk')$ to $f(\ell')$ gives $f(\ell)$, so
$G_{\eta_0}=G$.
Examine the composition map $\mu_G^{\mb{M}P}$:

$$\mu_G^{\mb{M}P}:\mb{M}P((G))=\bigotimes_{\mm{Vert}G}\mb{M}P((\mm{Leg}(v),b(v)))=\mb{M}P((\mm{Leg}(v),0))=$$ $$=\bigoplus_{K\in[\Gamma((\mm{Leg}(v),0))]}P((K))_{\mm{Aut}(K)}\longrightarrow\bigoplus_{L\in[\Gamma((3,1))]}P((L))_{\mm{Aut}(L)}=$$$$=\mb{M}P((3,1))$$

The map $\mu_G^{\mb{M}P}|_{P((H))_{\mm{Aut}(H)}}:P((H))_{\mm{Aut}(H)}\longrightarrow P((G))_{\mm{Aut}(G)}$ is the canonical projection $P((G))\longrightarrow P((G))_{\mm{Aut}(G)}$, since $P((G))=P((H))$, and $\mm{Aut}(H)=\mm{id}$.  The critical step here is replacing the vertex of $G\in\Gamma((\mm{Leg}(v),0))$ with the graph $G_{\eta_0}=H$, and gluing the legs $\ell$ and $\ell'$ via $f$ to obtain a graph homeomorphic to the original graph $G$.

\end{ex}

\begin{defn}
A cocycle D is a functor (not necessarily monoidal) $$\Gamma((n,b))\longrightarrow\mm{Graded\,\, Vector\,\, Spaces}$$ satisfying the following conditions:

 \begin{itemize}

 \item[i)]$\dim D(G)=1$ for all $G\in\Gamma((n,b))$
 \item[ii)] $D(*_{n,b})=\mb{C}$
 \item[iii)] For any morphism of stable graphs $f:G\longrightarrow G/I$, there is an isomorphism $$\nu_f:D(G/I)\otimes \bigotimes_{v\in\mm{Vert}(G/I)}D(f^{-1}(v))\longrightarrow D(G)$$

 \item[iv)] Given a composition of morphisms $G_0\overset{f_1}{\longrightarrow} G_1\overset{f_2}{\longrightarrow} G_2$, the following diagram commutes:
{\scriptsize\begin{center}
$\begin{diagram}
\node{D(G_2)\otimes\bigotimes_{v\in\mm{Vert}(G_2)}D(f_2^{-1}(v))\otimes\bigotimes_{v\in\mm{Vert}(G_1)}D(f_1^{-1}(v))} \arrow{s,l}{\simeq}\arrow{e,t}{\nu_{f_2}}\node{D(G_1)\otimes\bigotimes_{v\in\mm{Vert}(G_1)}D(f_1^{-1}(v))}\arrow[2]{s,r}{\nu_{f_1}}\\
\node{D(G_2)\otimes\bigotimes_{v\in\mm{Vert}(G_2)}(D(f_2^{-1}(v))\otimes\bigotimes_{w\in f_2^{-1}(v)}D(f_1^{-1}(w)))}\arrow{s,l}{D(G_2)\otimes\bigotimes_{v\in\mm{Vert}(G_2)}\nu_{f_1}|_{f_2^{-1}(v)}}\\
\node{D(G_2)\otimes\bigotimes_{v\in\mm{Vert}(G_2)}D((f_2\circ f_1)^{-1}(v))}\arrow{e,b}{\nu_{f_2\circ f_1}}\node{D(G_0)}
\end{diagram}$
\end{center}}

\item[v)] If $f:G_0\longrightarrow G_1$ is an isomorphism, the following diagram commutes:

\begin{center}$\begin{diagram}
\node{D(G_1)\otimes\bigotimes_{v\in\mm{Vert}(G_1)}D(f^{-1}(v))}\arrow{e,t}{\nu_f}\arrow{se,b}{\mm{id}\otimes g}\node{D(G_0)}\arrow{s,r}{D(f)}\\
\node[2]{D(G_1)}\end{diagram}$
\end{center}
where the map $g$ is given as follows.  Let $\pi_v$ be the projection $\pi_v:f^{-1}(v)\longrightarrow *_{n(v),b(v)}$.  Then $D(\pi_v)$ is the map $D(\pi_v):D(f^{-1}(v))\longrightarrow D(*_{n(v),b(v)})=\mb{C}$, and we set $g=\bigotimes_{\mm{Vert}(G_1)}D(\pi_v)$.
\end{itemize}
\end{defn}

Let $D$ be a cocycle.  Define $D^n(G):=D(G)^{\otimes n}$ and $D^{-1}(G):=D(G)^*$.  By definition, for any stable
graph $G$, $D(G)$ is one-dimensional and concentrated in a single degree, meaning $D^n(G)$ and $D^{-1}(G)$ are 
also one-dimensional and concentrated in a single degree. 

For example, let $V$ be a one-dimensional vector space over $\mb{C}$ thought of as a chain complex concentrated in degree 0, and assume $D$ is such that $D(G)=V[m]$ for some integer $m$.  Then $D^n(G)=V^{\otimes m}[nm]$ and $D^{-1}(G)=V^*[-m]$.

\begin{defn}\label{CocycleExample}
Define \begin{eqnarray}\label{DefOfK}
         \mc{K}(G) &:=& \mm{Det}(\mm{Edge}(G)) \\
          &=& {\bigwedge}^{\mm{top}}\mb{C}^{\mm{Edge}(G)}[|\mm{Edge}(G)|] \nn
       \end{eqnarray}
and
\begin{equation}\label{DefOfB}
  \mc{B}(G):=\mb{C}[b_1(G)]
\end{equation}


\end{defn}

\begin{defn}
Twisted modular operads are defined by simply replacing $P((G))$ in Definitions \eqref{Pre-Op} and \eqref{ModOpDef} with $D(G)\otimes P((G))$.  More specifically, the map \eqref{ModOp} becomes
\begin{equation}\label{TwistModOp}
    \mu:\mb{M}_DP((n,b)):=\bigoplus_{G\in[\Gamma((n,b))]}(D(G)\otimes P((G)))_{\mm{Aut}(G)}\longrightarrow P((n,b))
\end{equation}
and \eqref{mu} becomes
\begin{equation}\label{Twistmu}
    \mu_G:D(G)\otimes P((G))\longrightarrow D(*_{n,b})\otimes P((*_{n,b}))=P((n,b))
\end{equation}
A twisted modular operad should be thought of as a modular operad whose composition maps $\mu_G$ are twisted by a factor of $D(G)$.
\end{defn}

\begin{defn}
Let $s$ be a stable $\mb{S}$-module such that $\dim_{\mb{C}}s((n,b))=1$ for all $b$, $n$.
The \emph{coboundary} of $s$ is the cocycle $$D_s(G):=s((n,b))\otimes \bigotimes_{v\in\mm{Vert}(G)}s((n(v),b(v)))^*$$
\end{defn}

\begin{prop}
If $P$ is a modular $D$-operad, then $sP:=s\otimes P$ is a modular $D_sD$-operad.
\end{prop}
\begin{proof}
See \cite{Mod}.
\end{proof}

\begin{defn}
Let $D$ be a cocycle as defined above.  The free modular $D$-operad $\mb{M}_DP$ is defined by
\begin{equation}\label{TwistedFree}
    \mb{M}_DP((n,b))=\bigoplus_{G\in[\Gamma((n,b))]}(D(G)\otimes P((G)))_{\mm{Aut}(G)},
\end{equation}
with composition map $$D(G)\otimes\mb{M}_DP((G))\longrightarrow \mb{M}_DP((n,b)),$$
and is said to be \emph{generated} by the stable $\mb{S}$-module $P$.  The complex $\mb{M}_DP((G))$ is defined similarly to $\mb{M}P((G))$, i.e., $\mb{M}_DP((G))=\bigotimes_{\mm{Vert}(G)}\mb{M}_DP((\mm{Leg}(v),b(v)))$.
\end{defn}

The composition map is constructed in the same way as before by identifying the summand
$\bigotimes_{\mathrm{Vert}(G)}(D(G_{v_j})\otimes P((G_{v_j})))_{\mathrm{Aut}(G_{v_j})}$
with $(D(\tilde{G})\otimes P((\tilde{G})))_{\mathrm{Aut}(\tilde{G})}$, where $\tilde{G}$ is obtained from $G$ by replacing
each vertex $v$ of $G$ with $G_{v_j}$.


\begin{defn}
Let $P$ be a modular $D$-operad and define the cocycle $D^{\vee}:=\mc{K}D^{-1}$, where $\mc{K}$ is as in Definition \ref{CocycleExample}.  The \emph{Feynman Transform} of $P$ is the twisted modular $D^{\vee}$-operad given by

\begin{equation}\label{FeynmanDef}
    \mc{F}_DP=\mb{M}_{D^{\vee}}P^*,
\end{equation}
where $P^*$ is the linear dual of the stable $\mb{S}$-module. 
\end{defn}


\subsection{The Twisted Modular Operad $\mc{E}_V$}\label{evsings}

Let $V$ be a chain complex such that its homogeneous subspaces $V_i$ are finite-dimensional for all $i$, and let $|x|=$ deg $x$.  An inner product on $V$ is a non-degenerate bilinear form $B$ such that $B(dx,y)+(-1)^{|x|}B(x,dy)=0$, where $d$ is the differential of $V$.  Such a bilinear form is symmetric (resp. anti-symmetric) if $B(y,x)=(-1)^{|x||y|}B(x,y)$ (resp. $B(y,x)=-(-1)^{|x||y|}B(x,y)$), and has degree $k$ if $B(x,y)=0$ unless $|x|+|y|=k$.  In this paper we take $B$ to be anti-symmetric.  We will also write $V^S$ for $\bigoplus_SV$, where $V$ is as above and $S$ is a finite set. 


Let $D(G)= \mc{K}^{-1}(G)\otimes\bigotimes_{\mm{Edge}(G)}{\bigwedge}^2\mb{C}^{\{s_e,t_e\}}$, where, for a finite set $S$, $\mb{C}^S$ is a $|S|$-dimensional vector space over $\mb{C}$ with basis vectors indexed by $S$.  Set $\mc{E}_V((G)):=V^{\otimes\mm{Flag}G}$ with $v_1\otimes\cdots\otimes v_n\in \mc{E}_V((n,b))$ sitting in degree $n-\sum\limits_{i=1}^n\deg v_i$.  The degree adjustment from
 $\sum\limits_{i=1}^n\deg v_i$ is necessary because the complexes we are considering here are cohomological.  Define$$\mu_G^{\mc{E}_V}:D(G)
\otimes\mc{E}_V((G))\longrightarrow\mc{E}_V((n,b))$$by
\begin{equation}\label{dummy}
  ({\bigwedge}^{\epsilon_G}e_i)[-\epsilon_G]\otimes\bigotimes_{\mm{Edge}(G)}(f\wedge f')\otimes\bigotimes_{f\in\mm{Flag}(G)}v_f\mapsto (-1)^{\epsilon}\prod_{\substack{\mm{Edges}\\(f,f')}}B(v_f,v_{f'})\bigotimes_{f\in\mm{Leg}(G)}v_f,
\end{equation}
where $B$ is a pairing $V\otimes V\longrightarrow\mb{C}$ of degree 1, $\epsilon_G=|\mm{Edge}(G)|$, the factor $({\bigwedge}^{\epsilon_G}e_i)[-\epsilon_G]$ accounts for the degree of this pairing, which will be denoted by $B$, and the second factor $\bigotimes_{\mm{Edge}(G)}(f\wedge f')$ accounts for the anti-symmetry of $B$ for $(ff')$ a two-cycle of the involution of $G$.  This pairing will be denoted by $B$ for the remainder of the paper.  The sign $(-1)^{\epsilon}$ is defined as follows:

By Proposition 2.23 of \cite{Mod}, all compositions $\mu_G^{\mc{E}_V}:D(G)\otimes\mc{E}_V((G))\longrightarrow \mc{E}_V((n,b))$ can be built from compositions parameterized by contracting either the single edge of $G_0=*_{n_1,b_1}\coprod *_{n_2,b_2}$, or the single edge of $G_1=*^1_{n,b}$.

For the graph $ *_{n_1,b_1}\coprod *_{n_2,b_2} $ with edge given by the two-cycle $(ff')$, one has  $\mc{E}_V((G))=\mc{E}_V((n_1,b_1))\otimes \mc{E}_V((n_2,b_2))$ by Definition \ref{monad}, and the simple tensors have the form
$$(v_1\otimes\cdots\otimes v_f\otimes\cdots\otimes v_{n_1})\otimes(w_1\otimes\cdots\otimes w_{f'}\otimes\cdots\otimes w_{n_2})$$
Contraction via $B$ is done in several steps, each involving a permutation of the factors with the goal being to insert, in the operadic sense, $v$ into $w$.  Assume $v_f$ sits in the $\alpha^{th}$ spot of $v_1\otimes\cdots\otimes v_{n_1}$ and $w_{f'}$ sits in the $\beta^{th}$ spot of $w_1\otimes\cdots\otimes w_{n_2}$.  We adopt the supercommutativity conventions $x\otimes y=(-1)^{|x||y|}y\otimes x$ and
$(f\otimes g)(x\otimes y) = (-1)^{|g||x|}f(x)\otimes g(y)$.

First, move $v_f$ through $v_{f+1}\otimes\cdots\otimes v_{n_1}$.  This produces the sign $(-1)^{|v_f|(\sum\limits_{i=\alpha+1}^{n_1}|v_i|)}$.  Then move $w_i$ $1\leq i\leq \beta-1$ successively all the way to the right.  Moving $w_i$ gives the sign $(-1)^{|w_i|(|w|-|w_i|)}$.  The product now has the form
\begin{eqnarray*}
   && (-1)^{|v_f|(\sum\limits_{i=\alpha+1}^{n_1}|v_i|)}(-1)^{\sum\limits_{j=1}^{\beta-1}|w_j|(|w|-|w_j|)}v_1\otimes\cdots\otimes v_{n_1}\otimes v_f\otimes w_{f'}\otimes w_{f'+1}\otimes\cdots \\
   && \cdots\otimes w_{n_2}\otimes w_1\otimes\cdots\otimes w_{f'-1}
\end{eqnarray*}
Moving $\mm{id}\otimes\cdots\otimes B\otimes\cdots\otimes \mm{id}$ through the product from left to right gives $(-1)^{|B|(|v|-|v_{f}|)}$.
Finally, move $w_1,\dots,w_{\beta-1}$ all the way to the left.  Moving each $w_l$ for $1\leq l\leq \beta-1$ gives the sign $(-1)^{|w_{\ell}|(|v|+|w|-|v_f|-|w_{f'}|-|w_{\ell}|)}$.  The end result is
\begin{eqnarray}\label{signsEV}
   && (-1)^{|v_f|(\sum\limits_{i=\alpha+1}^{n_1}|v_i|)+\sum\limits_{j=1}^{\beta-1}|w_j|(|w|-|w_j|)+|B|(|v|-|v_{f}|)+\sum\limits_{\ell=1}^{\beta-1}|w_{\ell}|(|v|+|w|-|v_f|-|w_{f'}|-|w_{\ell}|)} \\
   && B(v_f,w_{f'})w_1\otimes\cdots\otimes w_{\beta-1}\otimes v_1\otimes\cdots\otimes v_{n_1}\otimes w_{\beta+1}\otimes\cdots\otimes w_{n_2}, \nn
\end{eqnarray}
which we write as $v\circ_i w$, where $v = v_1\otimes\cdots\otimes v_f\otimes\cdots\otimes v_{n_1}$ and $w = w_1\otimes\cdots\otimes w_{f'}\otimes\cdots\otimes w_{n_2}$.
Now let $G=*_{n,b}^1$.  Then
$\mu_G^{\mc{E}_V}:D(G)\otimes\mc{E}_V((G))\longrightarrow \mc{E}_V((n,b+1))$ is given by
\begin{eqnarray}\label{signsEV2}
   && v_1\otimes\cdots\otimes v_f\otimes\cdots\otimes v_{f'}\otimes\cdots\otimes v_n \nn\\
   &\mapsto& (-1)^{|v_f|(\sum\limits_{i=\alpha+1}^{\beta-1}|v_i|)+|B|((\sum\limits_{j=1}^{\beta-1}|v_j|)-|v_f|)}B(v_f,v_{f'})v_1\otimes\cdots\otimes v_{f-1}\otimes v_{f+1}\otimes\cdots \nn\\
   && \cdots\otimes v_{f'-1}\otimes v_{f'+1}\otimes\cdots\otimes v_n
\end{eqnarray}

\section{Tropical Morse Graphs}

In this section we extend the definition of tropical Morse trees as defined by Abouzaid, Gross, and Seibert in \cite{Clay}, to that of tropical Morse graphs. 

Given an integral affine manifold $\ms{B}$ of dimension $r$, we form the torus bundle $\tau: X(\ms{B})\to \ms{B}$ by factoring the tangent bundle $T\ms{B}$ by the local system $\Lambda$, where for $p\in \ms{B}$, the stalk $\Lambda_p$ of $\Lambda$ is locally generated by $\partial/\partial y_1,\dots,\partial/\partial y_n$, where $y_1,\dots,y_n$ are the affine coordinates of $\ms{B}$.  So in our 1-dimensional case in which $\ms{B}=\mb{R}/d\mb{Z}$, $X(\ms{B})$ is obtained by factoring each fiber $\mb{R}$ of $T\ms{B}$ by the lattice $\mb{Z}\cdot\partial/\partial y$.  See \cite{Clay} for a full account.

\begin{defn}\label{DefOfBZ}
$\ms{B}(\frac{1}{n}\mb{Z})$ is the set of points of the affine manifold $\ms{B}=\mb{R}/d\mb{Z}$ with coordinates in $\frac{1}{n}\mb{Z}$.
\end{defn}

In this section we define tropical Morse graphs and use them to construct Riemann surfaces with boundary on the elliptic curve $X(\ms{B})$.  Each boundary component lies in a finite union of Lagrangian submanifolds $\{L_{n_i}|n_i\in\mb{Z}\}$ of $X(\ms{B})$, the successive intersections of which comprise the corners $\{p_{i,j}\}$ of the Riemann surface.  The indices comprising the pair $(i,j)$ for each $p_{i,j}$ are ordered according to a boundary traversal defined by the surface lying on the left, where $i$ corresponds to $L_{n_i}$.  Note that in the case of a tree, this corresponds to a counterclockwise traversal, and in the case of a graph $G$ with $\mm{H}_1(G)=1$, this corresponds to a counterclockwise traversal of the outer boundary and a clockwise traversal of the inner boundary.  Each moduli space of tropical Morse graphs is indexed by a set of these corners, each being a lift of a point $p\in \ms{B}(\frac{1}{n_j-n_i}\mb{Z})$.

The definition of a \emph{tropical Morse graph} contains as a special case the definition of a \emph{tropical Morse tree} \cite{Clay}, the definition of which builds up on the notion of a \emph{rooted ribbon tree}.  For the sake of clarity, we reproduce the definition of a rooted ribbon tree nearly verbatim from \cite{Ar1}.
\begin{defn}\label{rootedribbontree}
A \emph{rooted ribbon tree} $\mathcal{R}$ is a connected tree with a finite number of vertices and edges, with no divalent vertices, together with the additional data of a cyclic ordering of edges at each vertex and a distinguished vertex referred to as the \emph{rooted vertex}.  We call the univalent vertices of $\mathcal{R}$ \emph{external} and the other vertices \emph{internal}, and insist that the root is an external vertex.  We orient all edges towards the root vertex, referring to the root vertex as the \emph{outgoing vertex} and all other external vertices as \emph{incoming vertices}.  We refer to edges adjacent to a vertex $v$ as \emph{incoming edges} if the assigned direction points towards $v$ and \emph{outgoing edges} otherwise.  Let $\mathcal{R}$ be a rooted ribbon tree with $d+1$ external vertices, for $d>0$.  There is a unique isotopy class of embeddings of $\mathcal{R}$ into the unit disk $D\subset \mb{R}^2$ such that each external vertex maps to $S^1\subset D$.  This embedding divides $D$ into $d+1$ regions, each of which meets the boundary $S^1$ in a segment.  We label the regions by $0,\dots,d$, proceeding counterclockwise around $S^1$ from the root vertex.  Given distinct integers $n_0,\dots,n_d$, assign to each edge $e$ the integer $n_e:=n_j-n_i$, called the \emph{acceleration} of $e$, where $e$ lies between regions $i$ and $j$.  A ribbon tree whose edges are labelled with accelerations is called \emph{decorated}.  We also label each external vertex of a decorated ribbon tree by $v_{ij}$, where the unique edge incident to the vertex lies between regions $i$ and $j$, and we take the indices $i$ and $j$ modulo $d+1$.  Note that the vertex labels are $v_{01},v_{12},\dots,v_{d-1,d},v_{d,0}$.
\end{defn}

\begin{defn}\label{TMGDef2}
Let $\mathcal{R}$ be a ribbon tree with $d+1$ external vertices, $d>0$, and $n_0,\dots,n_d\in \mb{Z}$ be the set of integers describing a decoration on $\mathcal{R}$.  Identify each edge $e$ of $\mathcal{R}$ with $[0,1]$ with coordinate $s$ and the orientation on $e$ pointing from $0$ to $1$.  A \emph{tropical Morse tree} is a map $\phi:\mathcal{R}\to S^1$ satisfying the following:

\begin{itemize}
\item[(1)] For any external vertex $v_{ij}$ of $\mathcal{R}$,
\begin{equation*}
p_{ij}:=\phi(v_{ij})\in S^1\left(\frac{1}{n_j-n_i}\mb{Z}\right)
\end{equation*}
\item[(2)]  For an edge $e$ of $\mathcal{R}$, $\phi(e)$ is either an affine line segment or a point in $S^1$. 
\item[(3)] For each edge $e$, there is a section $\tb{v}_e\in\Gamma(e,(\phi|_e)^*TS^1)$, called \emph{velocity of $e$}, satisfying 
\begin{enumerate}
\item[i.)] $\tb{v}_e(v)=0$ for each external vertex $v$ adjacent to the edge $e$.
\item[ii.)] For each edge $e\simeq [0,1]$ and $s\in[0,1]$, we have $\tb{v}_e(s)$ is tangent to $\phi(e)$ at $\phi(s)$, pointing in the same direction as the orientation on $\phi(e)$ induced by that on $e$.  By identifying $(\phi|_e)^*TS^1$ with the trivial bundle over $e$ using the affine structure on $S^1$, we have 
\begin{equation*}
\frac{d}{ds}\tb{v}_e(s)=n_e\phi_*\frac{\partial}{\partial s}
\end{equation*}
\item[iii.)] For any internal vertex $v$ of $\mathcal{R}$ the following balancing condition holds.  Let $e_1,\dots,e_p$ be the incoming edges and let $e_{\mm{out}}$ be the outgoing edge adjacent to $v$.  Then 
\begin{equation}
\tb{v}_{e_{\mm{out}}}(v) = \sum_{i=1}^p\tb{v}_{e_i}(v)
\end{equation}
\end{enumerate}
\end{itemize}
\end{defn}

A \emph{metric ribbon graph} $G$ is a connected graph with a finite number of vertices and edges, with no divalent vertices, together with the additional data of a cyclic ordering of edges at each vertex, but unlike rooted ribbon trees, there is no distinguished vertex.  The notions of \emph{external} and \emph{internal} edges are identical to those of Definition \ref{rootedribbontree}.  The external verticies are organized into $b+1$ cycles $\{\s_k\}$ with $\s:=\s_1\cdots\s_{b+1}$.  This unique factorization of $\s$ into cycles defines the topology of the corresponding Riemann surface, where each cycle $\s_k$ corresponds to a boundary of that surface.  For each $k$, $1\leq k\leq b+1$, each vertex comprising $\s_k$ maps to a copy of $S^1$, one copy for each $k$.

\begin{defn}\label{TMGDef}
Let $ G_{n,b,\s}^{\mm{trop}}(p_{1,2},p_{2,3},\dots,p_{n,1})$ denote the set of continuous maps $\phi:G\longrightarrow \ms{B}$ from metric ribbon graphs $G$ of genus $b$ and $n$ external legs to the affine manifold $\ms{B}$.  Each leg is labeled by $n_j-n_i$ for $n_i,n_j\in\mb{Z}$ and each internal edge $e$ is labeled by an integer $n_e$.  The endpoints of the legs are labeled by points $p_{i,j}$ as described above.  Each graph is decorated with a collection of affine displacement vectors $\tb{v}_e\in\Gamma(e,(\phi|_e)^*T\ms{B})$, satisfying the following properties:

\begin{itemize}
\item[(0)] The points $p_{ij}$ are organized into $b+1$ cycles $\{\s_k\}$ with $\s:=\s_1\cdots\s_{b+1}$.  This unique factorization of $\s$ into cycles defines the topology of the corresponding Riemann surface. 
\item[(1)] For any external vertex $v_{ij}$ of $G$, $$p_{ij}:=\phi(v_{ij})\in \ms{B}\left(\frac{1}{n_j-n_i}\mb{Z}\right)$$
\item[(2)] For an edge $e$ of $G$, $\phi(e)$ is either an affine line segment or a point in $\ms{B}$.
\item[(3)] For an edge $e$, there is a section $\tb{v}_e\in\Gamma(e,(\phi|_e)^*T\ms{B})$, called \emph{velocity of} $e$, satisfying  
\begin{enumerate}
\item[i.)] $\tb{v}_e(v)=0$ for each external vertex $v$ adjacent to the edge $e$. 
\item[ii.)] For each edge $e\simeq[0,1]$ and $s\in[0,1]$, we have $\tb{v}_e(s)$ is tangent to $\phi(e)$ at $\phi(s)$, pointing in the same direction as the orientation on $\phi(e)$ induced by that on $e$.  By identifying $(\phi|_e)^*T\ms{B}$ with the trivial bundle over $e$ using the affine structure on $\ms{B}$, we have $$\frac{d}{ds}\tb{v}_e(s)=n_e\phi_*\frac{\partial}{\partial s},$$ where $n_e\in\mb{Z}$. 
\begin{rem}
If $\ms{B}=\mb{R}^n/M$ for some lattice $M\subseteq\mb{Z}^n$, then this equation can be solved as $\tb{v}_e(s)=\tb{v}_e(0) +n_e(\phi(s)-\phi(0))\mod M$.
\end{rem}
\item[iii.)] Fix an orientation on $G$.  For any internal vertex $v$ of $G$ the following balancing condition holds.  Let $e_1,\dots,e_p$ be incoming edges adjacent to $v$ and let $f_1,\dots,f_q$ be the outgoing edges adjacent to $v$.  Then $$\sum_{i=1}^p\tb{v}_{e_i}(v)=\sum_{j=1}^q\tb{v}_{f_j}(v)$$
\item[iv.)] The length of an edge $e$ in $G$ coincides with $$\frac{1}{n_e}\log\left(\frac{\tb{v}_e(1)}{\tb{v}_e(0)}\right),$$ where each external edge has infinite length.  Since $\tb{v}_e(0)$ and $\tb{v}_e(1)$ are proportional vectors pointing in the same direction, their quotient make sense as a positive number.  There is one special case:  if $e$ is an external edge that is contracted by $\phi$, then $\tb{v}_e(0)=\tb{v}_e(1)=0$, but we still take the length to be infinite. 
\end{enumerate}
\end{itemize}

In this paper all external legs are oriented inward.  We have $G_{n,0,\s}^{\mm{trop}}=S_{n-1}^{\mm{trop}}$, where $\s$ is a single cycle and $S_{n-1}^{\mm{trop}}$ is the space of tropical Morse trees as defined in \cite{Clay}.
\end{defn}
\begin{rem}Internal edges are not allowed to collapse.  More specifically, if $\phi:G\longrightarrow \ms{B}$ is such that the balancing condition at a certain vertex forces $\phi|_e$ to be constant, then the domain $G$ is redefined by contracting $e$.  We also require at least one of the $n_e$ for $e$ external to be nonzero.  The reason for this is explained Section \ref{modulispacedegenerationsection}.
\end{rem}


\subsection{Holomorphic Polygons and Annuli}

\begin{defn}
Let $n\in\mb{Z}$, $\ms{B}=\mb{R}/d\mb{Z}$, and $X(\ms{B})=T\ms{B}/\Lambda$.  Define the section $\s_n:\ms{B}\longrightarrow X(\ms{B})$ via the local formula $$\s_n(y)=(y,-ny)$$
The Lagrangians $L_n$ of the torus bundle $\tau: X(\ms{B})\to \ms{B}$ are defined as the images of $\ms{B}$ under the sections $\s_n$.  Note the Lagrangians $L_n$ can be generalized to any integral affine manifold $\ms{B}$.

\end{defn}



Let $G$ be a Tropical Morse Graph and let $e\in \mm{Edge}(G)$.  As stated in part 4 of Definition \ref{TMGDef}, the vector $\tb{v}_e(s)$ has the form $\tb{v}_e(s)=\tb{v}_e(0)+n_e(\phi(s)-\phi(v))$.  If $e$ is external, then $\tb{v}_e(0)=0$ and $\phi(v)=p_{i,j}$.  Since $p_{i,j}\in \ms{B}(\frac{1}{n_e}\mb{Z})$, we have $n_e p_{i,j}\in\mb{Z}$ and this gives $\tb{v}_e(s)=n_e\phi(s)\mod\Lambda$.  Call this last equality the $Lagrangian$ $condition$ for the edge $e$.  In the tree case, the balancing conditions force $\textbf{v}_e(s)=n_e\phi(s)\mod\Lambda$ for the internal edges.  Intuitively, this is the statement that the tail and head of $\tb{v}_e(s)=\tb{v}_e(0)+(n_j-n_i)\phi(s)$ sit on $L_{n_i}$ and $L_{n_j}$, respectively.

In the graph case, one must impose the Lagrangian condition on an arbitrary edge within each generator of $\mm{H}_1(G)$.  The balancing conditions then force the condition on each edge comprising each generator.

Before we construct our holomorphic disks and annuli from tropical Morse graphs in detail, it is necessary for us to understand how these graphs are labeled.

\subsubsection{Labeling of Tropical Morse Graphs}


\begin{figure}[h]
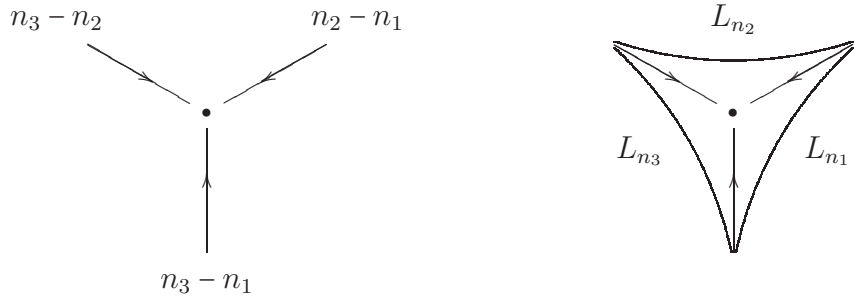

\[  \xygraph{
!{<0cm,0cm>;<1cm,0cm>:<0cm,1cm>::}
!{(0,0)}*+{_{\bullet}}="a"
!{(1.7321,1)}*+{}="b"
!{(0.5774,0.3333)}*+{}="a_mid"
!{(-1.7321,1)}*+{}="c"
!{(-0.5774,0.3333)}*+{}="c_mid"
!{(0,-2)}*+{}="d"
!{(0,-0.6666)}*+{}="d_mid"
!{(-7,0)}*+{_{\bullet}}="aa"
!{(-6.4226,0.3333)}*+{}="aa_mid"
!{(-5.2679,1)}*+{}="bb"
!{(-8.7321,1)}*+{}="cc"
!{(-7.5774,0.3333)}*+{}="cc_mid"
!{(-7,-2)}*+{}="dd"
!{(-7,-0.6666)}*+{}="dd_mid"
!{(-5,1.25)}*+{n_2-n_1}="e"
!{(-9,1.25)}*+{n_3-n_2}="f"
!{(-7,-2.25)}*+{n_3-n_1}="g"
!{(1.25,-0.5)}*+{L_{n_1}}="ee"
!{(0,1.25)}*+{L_{n_2}}="ff"
!{(-1.25,-0.5)}*+{L_{n_3}}="gg"
"b":"a_mid" "c":"c_mid" "d":"d_mid" "bb":"aa_mid" "cc":"cc_mid" "dd":"dd_mid"
"a"-"b" "a"-"c" "a"-"d"
"b"-@/_0.3cm/"d" "b"-@/^0.3cm/"c" "c"-@/^0.3cm/"d"
"aa"-"bb" "aa"-"cc" "aa"-"dd"
}  \]
\caption{The left-hand figure is the domain of the tropical Morse graph $u:G\longrightarrow B$ and the right-hand
figure is its fatgraph.  We see that as the boundary of the fatgraph is traversed in the counterclockwise direction,
the interior of the
fatgraph lies on the left, and the Lagrangians are encountered in the order $(L_{n_j}$, $L_{n_i})$, where the corresponding
edge on the left is labeled $n_j-n_i$.  If, or example, the upper right-hand edge were oriented outward, the labeling
of the edge would change to $n_1-n_2$.  The same holds for the remaining two edges.}
\label{fatgraphlabeling}
\end{figure}


\begin{defn} The fatgraph of a graph $G$ is a thickening of the edges and legs of $G$.  One replaces the edges with rectangles and glues them according to the given cyclic order at each vertex to obtain a surface of two real dimensions.
\end{defn}\label{fatgraphorientation}

For the sake of understanding the labeling of the edges of a tropical Morse graph, consider its fatgraph.  The fatgraph is homeomorphic to the Riemann surface with boundary to be constructed with the boundaries contained in Lagrangians.  With this in mind, label the boundaries of the fatgraph by $\{L_{n_i}\}$ for integers $\{n_i\}$ as follows.

Let $G$ be the domain of a tropical Morse graph.  Orient the boundary of the corresponding fatgraph so that as the boundary is traversed, the interior lies on the left.  For example, if the fatgraph is a disk then this is the counterclockwise orientation.  Each external vertex bounding an external leg corresponds to a point of intersection between two Lagrangians.  This means that this external vertex is mapped to a point of intersection between two Lagrangians in $X(\ms{B})$.  The legs therefore partition the boundary of the fatgraph into segments, each labeled by a Lagrangian $L_{n_j}$ for some $j$.  If an external vertex $v$, labeled by $p_i$, is a transition point from $L_{n_i}$ to $L_{n_j}$, and the relevant leg (external edge) is oriented away from $v$, then label that leg with $n_j-n_i$.  If the leg is oriented toward $v$, then label the leg with $n_i-n_j$.  See Figure \ref{fatgraphlabeling} for an illustration of this labeling scheme.

Similarly, label the internal edges according to which boundary segments lie on either side of the 2-dimensional strip defined by a given edge.  Let $e$ be such an edge bounded by segments labeled $L_{n_i}$ and $L_{n_j}$.  As illustrated in Figure \ref{pinch}, pinch $L_{n_i}$ and $L_{n_j}$ to some point on the interior of $e$, and label each half-edge according to the rule for legs.  This is a well-defined labeling of $e$ since the edge maintains its orientation through the pinch, and therefore flows toward the pinch on one side, and away from the pinch on the other.

\begin{figure}[h]
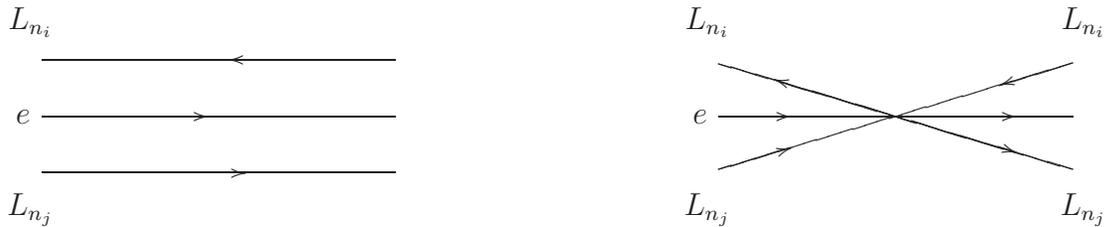

\[  \xygraph{
!{<0cm,0cm>;<1cm,0cm>:<0cm,1cm>::}
!{(-5,-0.75)}*+{}="a"
!{(0,-0.75)}*+{}="b"
!{(-2,-0.75)}*+{}="m" 
!{(-5,-1.25)}*+{L_{n_j}}="g"
!{(-5,0.75)}*+{}="c"
!{(0,0.75)}*+{}="d"
!{(-5,1.25)}*+{L_{n_i}}="h"
!{(-2.5,0.75)}*+{}="l"
!{(-5,0)}*+{}="e"
!{(0,0)}*+{}="f"
!{(-5,0)}*+{}="j"
!{(-2.5,0)}*+{}="k"
!{(-5.1,0)}*+{e}="i"
!{(4,0.75)}*+{}="aa"
!{(4,-0.75)}*+{}="cc"
!{(5.25,-0.375)}*+{}="gg" 
!{(4.75,0.5265)}*+{}="hh" 
!{(8.25,-0.5265)}*+{}="ii" 
!{(7.75,0.375)}*+{}="jj" 
!{(9,-0.75)}*+{}="bb"
!{(9,0.75)}*+{}="dd"
!{(4,0)}*+{}="ee"
!{(9,0)}*+{}="ff"
!{(4,1.25)}*+{L_{n_i}}="p"
!{(9,1.25)}*+{L_{n_i}}="q"
!{(4,-1.25)}*+{L_{n_j}}="r"
!{(9,-1.25)}*+{L_{n_j}}="s"
!{(3.9,0)}*+{e}="s"
!{(5.25,0)}*+{}="t"
!{(8.25,0)}*+{}="u"
"aa"-"bb" "cc"-"dd" "j":"k" "d":"l" "a":"m" "cc":"gg"
"a"-"b" "c"-"d" "e"-"f" "ee"-"ff" "bb":"hh" "dd":"jj" "aa":"ii"
"ee":"t" "ee":"u"
}  \]
\caption{Pinching the strip yields two triangles meeting at a point.  The labeling procedure for legs described above applied to the left half-edge gives $n_j-n_i$, as this half-edge is directed toward the corner (pinch).  Applying the procedure to the right half-edge also gives $n_j-n_i$, as this half-edge is directed away from the corner.}
\label{pinch}
\end{figure}


\subsubsection{Construction of Riemann Surfaces with Boundary}\label{ConstructionOfRS}

The surfaces, including their signs, are constructed from tropical Morse graphs in essentially the same way as holomorphic disks are constructed from tropical Morse trees.  The following construction can be found in the final chapter of \cite{Clay}.

A tropical Morse graph $\phi:G\longrightarrow \ms{B}$ defines a Riemann surface with boundary in the following way.

\begin{flushleft}Consider the map
\end{flushleft}

\begin{eqnarray}\label{holopolygon}
  R_e:e\times[0,1] &\longrightarrow& X(\ms{B})=T(\ms{B})/\Lambda \nn\\
  (s,t) &\mapsto& \s_{n_i}(\phi(s))-t\cdot\textbf{v}_e(s)
\end{eqnarray}
As the parameter $s$ runs through $e$, the vector $t\cdot\textbf{v}_e(s)$ sweeps out a two-dimensional region bounded by $L_{n_i},L_{n_j},\tau^{-1}(\phi(0)),$ and $\tau^{-1}(\phi(1))$.  This is illustrated in Figure \ref{tracing}.

\begin{figure}[h]
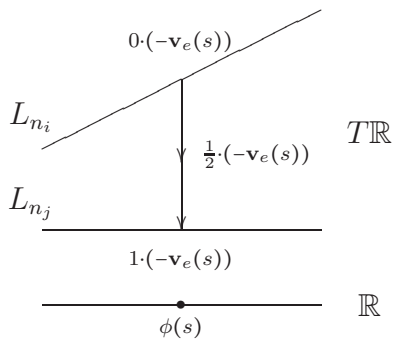

\[  \xygraph{
!{<0cm,0cm>;<1cm,0cm>:<0cm,1cm>::}
!{(-2,0) }*+{}="a"
!{(2,0) }*+{}="b"
!{(-2,1) }*+{}="c"
!{(2,3) }*+{}="d"
!{(0,2.15) }*+{}="e"
!{(0,-0.15) }*+{}="f"
!{(0,.75) }*+{}="g"
!{(-2,1.5) }*+{L_{n_i}}="h"
!{(-2,.35) }*+{L_{n_j}}="i"
!{(0,2.5) }*+{_{0\cdot(-\tb{v}_e(s))}}="j"
!{(1,1) }*+{_{\frac{1}{2}\cdot(-\tb{v}_e(s))}}="k"
!{(0,-.4) }*+{_{1\cdot(-\tb{v}_e(s))}}="l"
!{(-2,-1) }*+{}="m"
!{(2,-1) }*+{}="n"
!{(0,-1) }*+{_{\bullet}}="o"
!{(0,-1.3) }*+{_{\phi(s)}}="p"
!{(2.5,-1) }*+{\mb{R}}="q"
!{(2.5,1.3) }*+{T\mb{R}}="r"
"a"-"b" "c"-"d" "e":"f" "e":"g" "m"-"n"
}  \]
\caption{The tail of the vector is given by $R_e(s,0)=\s_{n_i}(\phi(s))\in L_{n_i} $, and the head
is given by $R_e(s,1)=\s_{n_j}(\phi(s))\in\L_{n_j}$.  See Section \ref{LConditionSection} for justification.}
\label{tracing}
\end{figure}



\begin{defn}\label{degree}
Let $(L_{n_i},L_{n_j})$ be an ordered pair of Lagrangians and let $p\in L_{n_i}\cap L_{n_j}$.  Set
\begin{equation}
    \deg p:=
    \begin{cases}
    1&\text{if $n_j<n_i$}\\
    0&\text{if $n_j>n_i$}
    \end{cases}
    \end{equation}
\end{defn}
Definition \ref{degree} is visualized in terms of Riemann surfaces as follows.  If $p$ is a degree 0 corner, then the surface widens in the direction given by the orientation of the leg bounded by $p$.  If $p$ is degree 1, then the rate at which the length of $\tb{v}_e$ changes, shrinks as $\phi(s)$ passes through $\phi(p)$.

Attached to each tropical Morse graph $u:G\longrightarrow \ms{B}$ is a sign $(-1)^{s(u)}$ defined by Abouzaid in \cite{Ab2} and a real number $\deg(u)$.  Write $p_{i,j}:=u(v_{i,j})$ for a degree-one point on the boundary of the fatgraph $u$ lying in the intersection of $L_{n_i}\cap L_{n_j}$, and such that $L_{n_j}$ immediately follows $L_{n_i}$ when $\partial u$ is traversed with the orientation given above.  Each degree-one point $p_{i,j}$ contributes a sign $(-1)^{s(p_{i,j})}$ as follows.  

Consider an epsilon neighborhood $B_{\epsilon}(p_{i,j})\subset X(\ms{B})$ of $p_{i,j}$, and a lift $\tilde{B}_{\epsilon}(p_{i,j})$ to $\mb{C}$, where the orientation on $\tilde{B}_{\epsilon}(p_{i,j})$ is induced by the natural orientation on $\mb{C}$.  The sign $(-1)^{s(p_{i,j})}$ is positive if the orientation on $L_{n_j}$ induced by the orientation on $\tilde{B}_{\epsilon}(p_{i,j})$ agrees with the given left-to-right orientation, and is negative otherwise.  Then
 \begin{equation}\label{polygonsign}
    (-1)^{s(u)}=\prod_{\{p_{i,j}|\deg p_{i,j}=1\}}(-1)^{s(p_{i,j})}
 \end{equation}
See Figure \ref{triangles} for the defining examples of how to calculate $(-1)^{s(u)}$.

For the degree, set
$$\deg(u) = \sum_{e\in\mm{Edge}(G)\cup\mm{Leg}(G)}\int\limits_e|\tb{v}_e(s)|ds,$$
where $\tb{v}_e\in\Gamma(e,(u|_e)^*T\ms{B})$.  This is the area of the fatgraph in $T\ms{B}$, taking into account any possible wrapping.

\begin{figure}[ht]
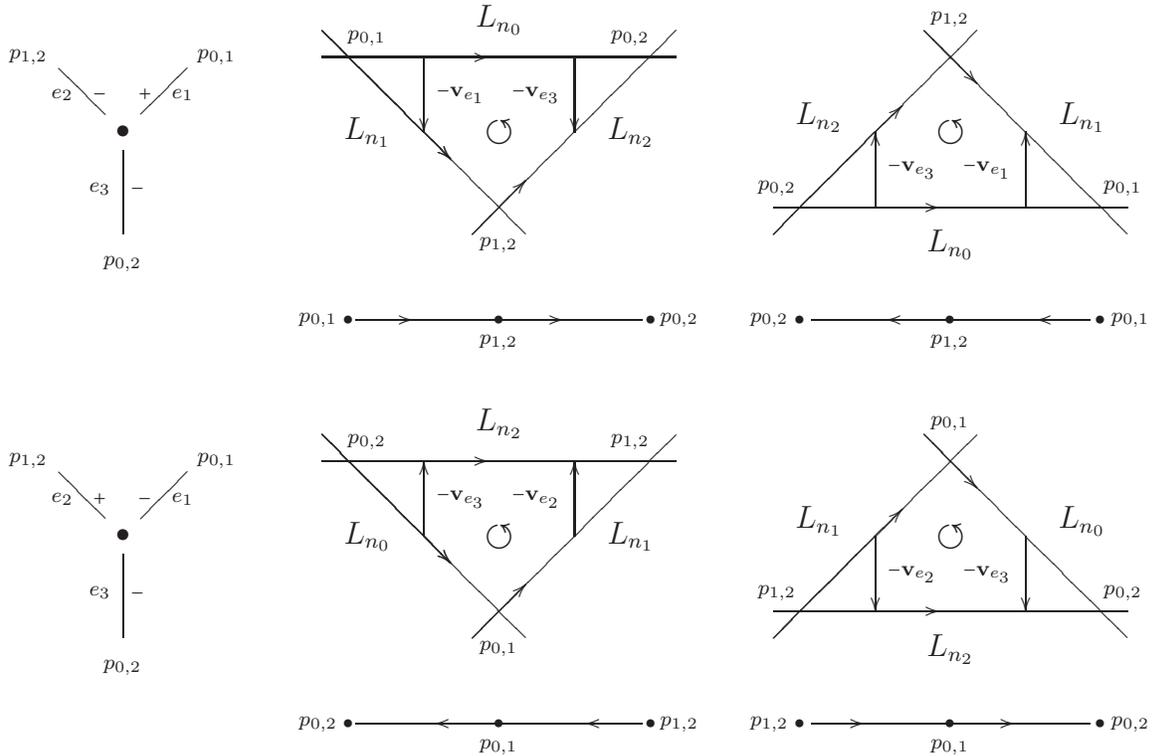

\[  \xygraph{
!{<0cm,0cm>;<1cm,0cm>:<0cm,1cm>::}
!{(-5.5,2.5) }*+{}="a"
!{(-2.5,-.5) }*+{}="b"
!{(-5.5,2) }*+{}="c"
!{(-.5,2) }*+{}="d"
!{(-3.5,-.5) }*+{}="e"
!{(-.5,2.5) }*+{}="f"
!{(-3.5,.5) }*+{}="g"
!{(-2.5,.5) }*+{}="h"
!{(-3,2) }*+{}="i"
!{(-4,2.15) }*+{}="k"
!{(-4,0.85) }*+{}="l"
!{(-2,2.15) }*+{}="m"
!{(-2,0.85) }*+{}="n"
!{(-1.25,1) }*+{L_{n_2}}="x"
!{(-4.75,1) }*+{L_{n_1}}="y"
!{(-3,2.5) }*+{L_{n_0}}="z"
!{(-4.75,2.25) }*+{_{p_{0,1}}}="oo"
!{(-1.25,2.25) }*+{_{p_{0,2}}}="pp"
!{(-3,-.5) }*+{_{p_{1,2}}}="qq"
!{(-5,-1.5) }*+{_{\bullet}}="q"
!{(-5.05,-1.5) }*+{}="qq"
!{(-.98,-1.5) }*+{_{\bullet}}="r"
!{(-.95,-1.5) }*+{}="rr"
!{(-5.4,-1.5) }*+{_{p_{0,1}}}="ss"
!{(-3,-1.8) }*+{_{p_{1,2}}}="tt"
!{(-.6,-1.5) }*+{_{p_{0,2}}}="uu"
!{(-3,-1.5) }*+{_{\bullet}}="s"
!{(-4,-1.5) }*+{}="t"
!{(-2,-1.5) }*+{}="u"
!{(-3.5,1.5) }*+{_{-\tb{v}_{e_1}}}="o"
!{(-2.5,1.5) }*+{_{-\tb{v}_{e_3}}}="p"
!{(-3,1) }*+{\circlearrowleft}="j"
!{(-8,1) }*+{\bullet}="aa"
!{(-9,2) }*+{}="bb"
!{(-7,2) }*+{}="cc"
!{(-8,-.5) }*+{}="dd"
!{(-6.75,2) }*+{_{p_{0,1}}}="ee"
!{(-9.25,2) }*+{_{p_{1,2}}}="ff"
!{(-8,-.75) }*+{_{p_{0,2}}}="gg"
!{(-8.3,.25) }*+{_{e_3}}="jj"
!{(-7.2,1.5) }*+{_{e_1}}="ii"
!{(-8.8,1.5) }*+{_{e_2}}="hh"
!{(-7.7,1.5) }*+{_+}="kk"
!{(-8.3,1.5) }*+{_-}="ll"
!{(-7.8,.25) }*+{_-}="mm"
!{(.5,0) }*+{}="az"
!{(5.5,0) }*+{}="bz"
!{(.5,-.5) }*+{}="cz"
!{(3.5,2.5) }*+{}="dz"
!{(2.5,2.5) }*+{}="ez"
!{(5.5,-.5) }*+{}="fz"
!{(2.5,1.5) }*+{}="gz"
!{(3.5,1.5) }*+{}="hz"
!{(3,0) }*+{}="iz"
!{(3,1) }*+{\circlearrowleft}="jz"
!{(2,1.15) }*+{}="rz"
!{(2,-0.15) }*+{}="sz"
!{(4,1.15) }*+{}="tz"
!{(4,-0.15) }*+{}="uz"
!{(2.5,.5) }*+{_{-\tb{v}_{e_3}}}="ggz"
!{(3.5,.5) }*+{_{-\tb{v}_{e_1}}}="hhz"
!{(1,-1.5) }*+{}="kz"
!{(5.0,-1.5) }*+{}="lz"
!{(1,-1.5) }*+{_{\bullet}}="mz"
!{(5,-1.5) }*+{_{\bullet}}="nz"
!{(3,-1.5) }*+{_{\bullet}}="oz"
!{(2,-1.5) }*+{}="pz"
!{(4,-1.5) }*+{}="qz"
!{(.6,-1.5) }*+{_{p_{0,2}}}="vz"
!{(3,-1.8) }*+{_{p_{1,2}}}="wz"
!{(5.4,-1.5) }*+{_{p_{0,1}}}="xz"
!{(1.25,1.2) }*+{L_{n_2}}="aaz"
!{(4.75,1.2) }*+{L_{n_1}}="bbz"
!{(3,-.5) }*+{L_{n_0}}="ccz"
!{(3,2.5) }*+{_{p_{1,2}}}="ddz"
!{(.7,.25) }*+{_{p_{0,2}}}="eez"
!{(5.3,.25) }*+{_{p_{0,1}}}="ffz"
"az"-"bz" "cz"-"dz" "ez"-"fz" "cz":"gz" "ez":"hz" "az":"iz" "lz":"qz" "oz":"pz" "kz"-"lz" "sz":"rz" "uz":"tz"
"a"-"b" "a":"g" "c"-"d" "e"-"f" "a":"g" "e":"h" "c":"i" "bb"-"aa" "cc"-"aa" "aa"-"dd"
"k":"l" "m":"n" "qq"-"rr" "qq":"t" "s":"u"
}  \]

\[  \xygraph{
!{<0cm,0cm>;<1cm,0cm>:<0cm,1cm>::}
!{(-5.5,2.5) }*+{}="a"
!{(-2.5,-.5) }*+{}="b"
!{(-5.5,2) }*+{}="c"
!{(-.5,2) }*+{}="d"
!{(-3.5,-.5) }*+{}="e"
!{(-.5,2.5) }*+{}="f"
!{(-3.5,.5) }*+{}="g"
!{(-2.5,.5) }*+{}="h"
!{(-3,2) }*+{}="i"
!{(-4,2.15) }*+{}="k"
!{(-4,0.85) }*+{}="l"
!{(-2,2.15) }*+{}="m"
!{(-2,0.85) }*+{}="n"
!{(-1.25,1) }*+{L_{n_1}}="x"
!{(-4.75,1) }*+{L_{n_0}}="y"
!{(-3,2.5) }*+{L_{n_2}}="z"
!{(-4.75,2.25) }*+{_{p_{0,2}}}="oo"
!{(-1.25,2.25) }*+{_{p_{1,2}}}="pp"
!{(-3,-.5) }*+{_{p_{0,1}}}="qq"
!{(-5,-1.5) }*+{_{\bullet}}="q"
!{(-5.05,-1.5) }*+{}="qq"
!{(-.98,-1.5) }*+{_{\bullet}}="r"
!{(-.95,-1.5) }*+{}="rr"
!{(-5.4,-1.5) }*+{_{p_{0,2}}}="ss"
!{(-3,-1.8) }*+{_{p_{0,1}}}="tt"
!{(-.6,-1.5) }*+{_{p_{1,2}}}="uu"
!{(-3,-1.5) }*+{_{\bullet}}="s"
!{(-4,-1.5) }*+{}="t"
!{(-2,-1.5) }*+{}="u"
!{(-3.5,1.5) }*+{_{-\tb{v}_{e_3}}}="o"
!{(-2.5,1.5) }*+{_{-\tb{v}_{e_2}}}="p"
!{(-3,1) }*+{\circlearrowleft}="j"
!{(-8,1) }*+{\bullet}="aa"
!{(-9,2) }*+{}="bb"
!{(-7,2) }*+{}="cc"
!{(-8,-.5) }*+{}="dd"
!{(-6.75,2) }*+{_{p_{0,1}}}="ee"
!{(-9.25,2) }*+{_{p_{1,2}}}="ff"
!{(-8,-.75) }*+{_{p_{0,2}}}="gg"
!{(-8.3,.25) }*+{_{e_3}}="jj"
!{(-7.2,1.5) }*+{_{e_1}}="ii"
!{(-8.8,1.5) }*+{_{e_2}}="hh"
!{(-7.7,1.5) }*+{_-}="kk"
!{(-8.3,1.5) }*+{_+}="ll"
!{(-7.8,.25) }*+{_-}="mm"
!{(.5,0) }*+{}="az"
!{(5.5,0) }*+{}="bz"
!{(.5,-.5) }*+{}="cz"
!{(3.5,2.5) }*+{}="dz"
!{(2.5,2.5) }*+{}="ez"
!{(5.5,-.5) }*+{}="fz"
!{(2.5,1.5) }*+{}="gz"
!{(3.5,1.5) }*+{}="hz"
!{(3,0) }*+{}="iz"
!{(3,1) }*+{\circlearrowleft}="jz"
!{(2,1.15) }*+{}="rz"
!{(2,-0.15) }*+{}="sz"
!{(4,1.15) }*+{}="tz"
!{(4,-0.15) }*+{}="uz"
!{(2.5,.5) }*+{_{-\tb{v}_{e_2}}}="ggz"
!{(3.5,.5) }*+{_{-\tb{v}_{e_3}}}="hhz"
!{(1,-1.5) }*+{}="kz"
!{(5.0,-1.5) }*+{}="lz"
!{(1,-1.5) }*+{_{\bullet}}="mz"
!{(5,-1.5) }*+{_{\bullet}}="nz"
!{(3,-1.5) }*+{_{\bullet}}="oz"
!{(2,-1.5) }*+{}="pz"
!{(4,-1.5) }*+{}="qz"
!{(.6,-1.5) }*+{_{p_{1,2}}}="vz"
!{(3,-1.8) }*+{_{p_{0,1}}}="wz"
!{(5.4,-1.5) }*+{_{p_{0,2}}}="xz"
!{(1.25,1.2) }*+{L_{n_1}}="aaz"
!{(4.75,1.2) }*+{L_{n_0}}="bbz"
!{(3,-.5) }*+{L_{n_2}}="ccz"
!{(3,2.5) }*+{_{p_{0,1}}}="ddz"
!{(.7,.25) }*+{_{p_{1,2}}}="eez"
!{(5.3,.25) }*+{_{p_{0,2}}}="ffz"
"az"-"bz" "cz"-"dz" "ez"-"fz" "cz":"gz" "ez":"hz" "az":"iz" "kz"-"lz" "rz":"sz" "tz":"uz" "mz":"pz" "oz":"qz"
"a"-"b" "a":"g" "c"-"d" "e"-"f" "a":"g" "e":"h" "c":"i" "bb"-"aa" "cc"-"aa" "aa"-"dd"
"l":"k" "n":"m" "qq"-"rr" "s":"t" "rr":"u" 
}  \]
\caption{The contributions, counterclockwise from upper left, are: $+1$, $+1$, $-1$, and $-1$.  Note that even if the outgoing edge $e$ of a vertex $v$ is such that $n_e>0$, the signs $(-1)^{s(u)}$ remain the same.  The only difference is the polygon may not close once $e$ terminates.  The arrow in the segment below each triangle indicates the direction of motion of $\phi(s)\in\mb{R}$.  The corner points $p_{i,j}$ are, as always, ordered according to a boundary traversal defined by the interior lying to the left, which, in this case, results in a counterclockwise ordering of the corners.}
\label{triangles}\end{figure}

\subsection{The Lagrangian Condition}\label{LConditionSection}
\begin{prop}\label{Lcondition}Let $G$ be trivalent.  Imposing the Lagrangian condition on a single edge of each generator of $H_1(G)$ is enough to guarantee the condition holds throughout $G$, as long the condition is imposed once and only once per generator.
\end{prop}
\begin{lem}\label{laglemma}Given an arbitrary vertex $v$ in $G$, if all but one of the attached edges have the Lagrangian condition, then so does the remaining edge.
\end{lem}
\begin{proof} Orient the attached edges with the Lagrangian condition toward $v$, and the remaining edge, labeled $e$, away from $v$.  Let $\{l_i\}$ be the subset of the attached edges $\mm{E}_v$ which are external legs, and let $\{e_j\}=\mm{E}_v-\{l_i\}-\{e\}$ be the remaining edges.  Note that by the definition of a tropical Morse graph, $\sum n_{l_i}+\sum n_{e_j}=n_e$.  The balancing condition at $v$ yields

\begin{eqnarray}
\tb{v}_e(v) &=& \sum\tb{v}_{l_i}(v)+\sum\tb{v}_{e_j}(v) \\
                  &=& \sum\tb{v}_{l_i}(0)+n_{l_i}(\phi(v)-p_{i,i-1})  + \sum\tb{v}_{e_j}(w_j)+n_{e_j}(\phi(v)-\phi(w_j)) \nn\\
                  &=& (\sum n_{l_i})\phi(v) -\sum n_{l_i}p_{i-1,i} +\sum\tb{v}_{e_j}(w_j)  +(\sum n_{e_j}) \phi(v)  -\sum n_{e_j}\phi(w_j) \nn \\
                  &=& (\sum n_{l_i}+\sum n_{e_j})\phi(v) + \sum(\tb{v}_{e_j}(w_j)-n_{e_j}\phi(w_j)) - \sum n_{l_i}p_{i-1,i}  \nn\\
                  &=& n_e\phi(v) +\sum(\tb{v}_{e_j}(v)-n_{e_j}\phi(v)) -\sum n_{l_i}p_{i-1,i} \mod\Lambda  \nn\\
                  &=& n_e\phi(v) \mod\Lambda\nn
\end{eqnarray}
where $e_j$ is bound by $v$ and $w_j$ and $\tb{v}_{e_j}(w_j)-n_{e_j}\phi(w_j)=\tb{v}_{e_j}(v)-n_{e_j}\phi(v)\mod\Lambda$ for all $e_j\in E_v-\{e\}$ by assumption.
\end{proof}
\begin{proof}(of Proposition \ref{Lcondition})

Let $S=\{c_i\}$ be the set of minimal generators of $H_1(G)$.  Say that an edge or external leg of $G$ is $marked$ if it has the Lagrangian condition.  Say that a generator is marked if at least one of its edges is marked.  Mark each segment comprising $\sum c_i$, and mark every external leg of $G$.  If a given $c_i$ is marked, unmark all but one of its marked edges.  Let $T\subseteq S$ be the set of marked generators, and repeat this process, but with $\sum_{S\setminus T}c_i$.  Continue until all $c_i$ are marked.  The result is that each $c_i$ is marked exactly once.
I claim there exists a vertex with no less than $n(v)-1$ marked edges.

Suppose the claim does not hold for a graph $G$.  For each vertex $v$, let $n_m(v)$ be the number of marked edges emanating from $v$.  Let $n$ be the number of external legs of $G$.  Each marked internal edge is bounded by two vertices, and all external legs are marked, so $$\sum_{\mm{Vert}(G)} n_m(v)=n+2b_1(G)$$  For the sake of contradiction, suppose $n_m(v)\leq n(v)-2$ for all $v\in\mm{Vert}(G)$.  Letting $E=|\mm{Edges}(G)|$ and $V=|\mm{Vert}(G)|$, one has

\begin{eqnarray}
  E-V+1 &=& b_1(G)\nn \\
   &=& \frac{1}{2}(\sum n_m(v)-n)\nn \\
   &\leq& \frac{1}{2}(\sum(n(v)-2)-n) \\
   &=& \frac{1}{2}(\sum n(v)-2V-n) \nn\\
   &=& \frac{1}{2}(2E+n-2V-n) \nn\\
   &=& E-V,\nn
\end{eqnarray}
a contradiction.

So there exists a vertex $v'$ such that $n_m(v')\geq n(v')-1$.  If $v'$ comprises part of a cycle, then it comprises part of a minimal generator.  If $n_m(v')=n(v')=3$, then this minimal generator is marked more than once, which is impossible.  This means that if $v'$ comprises part of a cycle, then $n_m(v')=2$.

If $v'$ does not comprise part of a cycle, then $n_m(v')=2$ as well.  Indeed, edges not comprising cycles are initially unmarked, so $v'$ must be attached to either 2 or 3 external legs.  If $v'$ is attached to 3 external legs, then $G$ must be the unique graph with a single vertex $v'$, no edges, and 3 legs, in which case the proposition is vacuously true.

I claim that upon finding and subsequently marking the remaining unmarked edges of $k$ vertices, each with $n_m(v)=2$, there exists another vertex $v$ with $n_m(v)=2$.  Suppose not.  Then there are $k$ vertices such that $n_m(v)=n(v)=3$ and $V-k$ vertices such that $n_m(v)$ is either 0 or 1.  Each newly marked internal edge counts twice in the sum $\sum_{\mm{Vert}(G)}n_m(v)$, giving \begin{eqnarray*}
                                                           2k+2b_1+n &=& \sum_{\mm{Vert}(G)}n_m(v) \\
                                                            &\leq& 3k+V-k \\
                                                            &=& V+2k,
                                                \end{eqnarray*}
giving $2b_1+n\leq V$.  But $2b_1-2+n=V$ by equation (2.10) of \cite{Mod}, so this is a contradiction.  The entire graph can therefore be inductively marked by finding vertices $v$ with $n_m(v)=2$ and using Lemma \ref{laglemma}.

\end{proof}

\subsection{The Dimension of The Moduli Space $G_{d,b,\s}^{\mm{trop}}$}



\begin{lem}\label{outputlemma}
Let $G$ be a tropical Morse graph, the external legs $l_i$ of which are oriented inward.  Label the endpoints of the external legs by points $p_{i,j}$ as in Definition \ref{TMGDef}.  Other than the external legs, no particular orientation of the internal edges is assumed.  Let $\ell$ be one of the external legs, bounded externally by a vertex $p$, and internally by a vertex $v$.  If $\phi$ maps the external vertices $v_i\neq p$ to $p_{i-1,i}$, then
\begin{equation}\label{output}
  n_{\ell}\phi(p)=\sum_{l_i\neq \ell}n_{l_i}p_{i-1,i}\mod d\mb{Z}
\end{equation}

\end{lem}

\begin{proof}

Let $e$ be an arbitrary edge (legs included) of $G$ oriented from a vertex $u$ to a vertex $w$, and let $n_e$ be the integer labeling $e$.  If $e$ is re-oriented from $w$ to $u$, then the integer label changes from $n_e$ to $-n_e$.

Let $u_k$ be an internal vertex and write the set of edges attached to $u_k$ as $\{e_i^k\}\sqcup\{f_j^k\}$, where the $e_i^k$ are oriented toward $u_k$ and the $f_j^k$ are oriented way from $u_k$.  The condition $\sum v_{\mm{in}}=\sum v_{\mm{out}}$ forces  $$\sum_i n_{e_i^k}+\sum_j-n_{f_j^k}=0$$
This is independent of the orientation.  Indeed, let $e$ be an edge oriented toward $u_k$ and let $e'$ be the same edge, but with the opposite orientation.  Then $n_{e'}=-n_e$ and \begin{eqnarray*}
                                                             \sum_{e_i^k\neq e}n_{e_i^k}+\left(\sum_{f_j^k}-n_{f_j^k}\right)-n_{e'} &=& \sum_{e_i^k\neq e}n_{e_i^k}+n_e+\sum_{f_j^k}-n_{f_j^k} \\
                                                              &=& \sum n_{e_i^k}+\sum-n_{f_j^k} \\
                                                              &=& 0
                                                           \end{eqnarray*}

The balancing condition at the vertex $u_k$ then takes the form $\sum \tb{v}_{e_i^k}(1)=\sum \tb{v}_{f_j^k}(0)$,
and summing over $\mm{Vert}(G)$ gives
\begin{equation}\label{Lagsum}
  n_{\ell}(\phi(v)-\phi(p))+\sum_k\sum_i\tb{v}_{e_i^k}(1)=\sum_k\sum_j\tb{v}_{f_j^k}(0)\mod d\mb{Z},
\end{equation}
where $\tb{v}_{\ell}(0)=0$ as $\ell$ is external.  Since $\tb{v}_e(0)=0$ for any external edge $e$, we may assume the outer sum on the right hand side runs over the internal vertices only.  As all external legs are oriented inward, we may also assume that each summand on the right hand side appears on the left hand side as a term comprising a vector $\tb{v}_{e_i^k}(1)$.  Indeed, each $u_k$ is connected, through an edge $f_j^k$ oriented away from $u_k$, to some other internal vertex $u_l$.  This means there is a summand on the left hand side of the form $\tb{v}_{e_i^l}(1)=\tb{v}_{f_j^k}(0)+n_{e_i^l}(\phi(u_l)-\phi(u_k))$.

Equation $\ref{Lagsum}$ takes the form
\begin{equation}\label{Lagsum2}
  n_{\ell}(\phi(v)-\phi(p))+\sum_k\sum_i n_{e_i^k}(\phi(u_k)-\phi(u_l))=0\mod d\mb{Z}
\end{equation}
If $u_{\ell}$ is external, then $\phi(u_{\ell})=p_{i,i-1}$ for some $i$.  Each vertex $u$ appears in the sum exactly $n(u)$ times, once for each edge attached to $u$.  Consider the set of edges $E_k$ attached to a vertex $u_k$.  Let $I_k\subseteq E_k$ be the subset of edges oriented toward $u_k$ and let $O_k=E_k\setminus I_k$ be the set of edges oriented away from $u_k$.

If $e_i^k=f_j^l$, then $n_{e_i^k}=n_{f_j^l}$ and each term $-n_{e_i^k}\phi(u_l)$ can also be written $-n_{f_j^l}\phi(u_l)$.  So we can regroup the terms of the sum $$n_{\ell}\phi(v)+\sum_k\sum_i n_{e_i^k}(\phi(u_k)-\phi(u_l))$$
to obtain
\begin{eqnarray*}
    && -\sum n_lp_{i-1,i}+\sum_k(\sum_{I_k}n_{e_i^k}\phi(u_k)-\sum_{O_k}n_{f_j^k}\phi(u_k)) \\
   &=& -\sum n_lp_{i-1,i}+ \sum_k(\sum_{I_k}n_{e_i^k}-\sum_{O_k}n_{f_j^k})\phi(u_k)\\
   &=& -\sum n_lp_{i-1,i}
\end{eqnarray*}
This gives $n_{\ell}\phi(p)=\sum n_lp_{i-1,i}\mod d\mb{Z}$.

\end{proof}

Lemma \ref{outputlemma} is best understood in terms of the polygon or annulus built from the given tropical Morse graph $G$.  In the notation of the lemma, if $\ell$ is an external leg labeled by $n_{\ell}\neq0$ and bounded externally by $p$, then $\phi(p)$ is determined by $\{\phi(p_i)|p_i\neq p,n_{\ell_i}\neq0\}$, where for each $i$, $p_i$ bounds the external leg $\ell_i$.

In terms of the associated polygon or annulus, this is the statement that the location of the corner $(\phi(p),-n_{\ell}\phi(p))\in T\ms{B}$ is determined up to a lattice shift.  Indeed, if one fixes the locations of all but one corner of a polygon or annulus and fixes the slope of each edge, then the location of the final corner is determined.

If $n_{\ell}=0$, then $\phi(p)$ unrestricted by Equation \ref{output}.  In our definition of tropical Morse graphs, $\phi(p)$ is determined if $\deg \ell=1$ and unrestricted if $\deg \ell=0$.  In either case, $\phi(p)$ is a marked point on $L\cap L$ for some Lagrangian $L$, and does not constitute a corner of the polygon or annulus.  In other words, the location of $\phi(p)$ is not determined by the geometry of the polygon or annulus, as it is in the case that $n_{\ell}\neq0$.



\begin{theorem}\label{DimensionTheorem}
$\dim G_{n,b,\s}^{\mm{trop}}=n-2+2b_1-\sum\deg p_i$, where $p_i$ labels the $i^{th}$ external leg.
\end{theorem}

\begin{proof}

Recall that a \emph{marked} edge is one that has the Lagrangian condition.  Let $S$ be the set of marked edges of $G$.  The balancing conditions and the Lagrangian conditions give rise to a
$$(b_1+|\mm{Vert}(G)|)\times (|\mm{Vert}(G)|+1+|\mm{Edge}(G)|)$$matrix $A_G$,
whose columns are indexed by the elements
$$\{\phi(v)|v\in\mm{Vert}(G)\}\cup\{\phi(p)\}\cup\{\tb{v}_e(0)|e\in\mm{Edge}(G)\},$$
and whose rows are indexed by $\mm{Vert}(G)\cup \{l_e|e\in S\}$.  So the columns of $A_G$ correspond to the variables of the moduli space $G_{n,b,\s}^{\mm{trop}}$ not fixed by the initial data, and the rows correspond to the constraints on those variables.  Since $b_1+|\mm{Vert}(G)|=|\mm{Edge}(G)|+1$, $A_G$ will have maximal rank if the
square submatrix given by the rightmost $|\mm{Edge}(G)|+1$ columns is invertible.  Denote this square matrix by $B_G$.

The first row of $B_G$ corresponds to the balancing condition at the unique internal vertex $u$ which is connected to $p$.  Order the remaining rows and all of the columns as follows.  From $u$, choose one of the edges $e$ connected to $u$, and traverse that edge until reaching the vertex which bounds it on its opposite end.  The balancing condition at this second vertex gives the second row, and the second column is indexed by $\tb{v}_e(0)$.  If $e$ comprises part of a loop in $\mm{H}_1(G)$, then continue in this way around the minimal loop containing $e$.  If not, then continue along an arbitrary path until it
terminates with a vertex which is connected to an external leg.  For the next vertex, backtrack from the current vertex until reaching the most recent junction, and move along the untraced edge to the next vertex.  This new edge marks the next column, and the new vertex marks the new row.  Continue in this way until all internal edges of $G$ have been traversed. The marked edges will be those traversed last within their respective minimal loops.


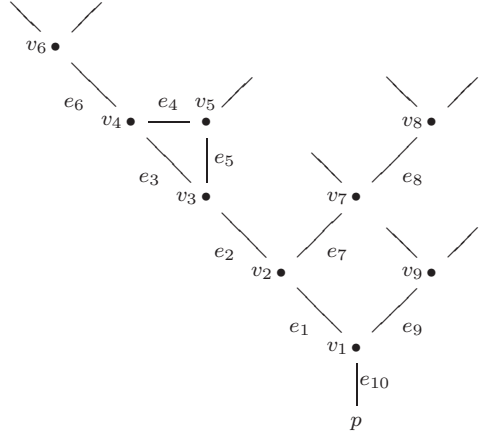
\begin{ex}
\end{ex}
\begin{figure}[H]
\[  \xygraph{
!{<0cm,0cm>;<1cm,0cm>:<0cm,1cm>::}
!{(0,0) }*+{_{\bullet}}="a"
!{(1,1) }*+{_{\bullet}}="b"
!{(1.75,1.75) }*+{}="c"
!{(0.25,1.75) }*+{}="d"
!{(-1,1) }*+{_{\bullet}}="e"
!{(0,2) }*+{_{\bullet}}="f"
!{(-.75,2.75) }*+{}="r"
!{(1,3) }*+{_{\bullet}}="g"
!{(0.25,3.75) }*+{}="h"
!{(1.75,3.75) }*+{}="i"
!{(-2,2) }*+{_{\bullet}}="j"
!{(-2,3) }*+{_{\bullet}}="k"
!{(-1.25,3.75) }*+{}="l"
!{(-3,3) }*+{_{\bullet}}="m"
!{(-4,4) }*+{_{\bullet}}="n"
!{(-3.25,4.75) }*+{}="o"
!{(-4.75,4.75) }*+{}="p"
!{(0,-1) }*+{_p}="q"
!{(0,-2) }*+{}="qq"
!{(-.25,0) }*+{_{v_1}}="aa"
!{(-1.25,1) }*+{_{v_2}}="bb"
!{(-2.25,2) }*+{_{v_3}}="cc"
!{(-3.25,3) }*+{_{v_4}}="dd"
!{(-2,3.25) }*+{_{v_5}}="ee"
!{(-4.25,4) }*+{_{v_6}}="ff"
!{(-.25,2) }*+{_{v_7}}="gg"
!{(.75,3) }*+{_{v_8}}="hh"
!{(.75,1) }*+{_{v_9}}="ii"
!{(-.75,.25) }*+{_{e_1}}="aaa"
!{(-1.75,1.25) }*+{_{e_2}}="bbb"
!{(-2.75,2.25) }*+{_{e_3}}="ccc"
!{(-2.5,3.25) }*+{_{e_4}}="ddd"
!{(-1.75,2.5) }*+{_{e_5}}="eee"
!{(-3.75,3.25) }*+{_{e_6}}="fff"
!{(-.25,1.25) }*+{_{e_7}}="ggg"
!{(.75,2.25) }*+{_{e_8}}="hhh"
!{(.75,.25) }*+{_{e_9}}="iii"
!{(.25,-.45) }*+{_{e_{10}}}="jjj"
!{(3.5,-.45) }*+{}="kkk"
"a"-"b" "b"-"d" "b"-"c" "a"-"q" "a"-"e" "e"-"f" "f"-"r" "f"-"g" "g"-"h" "g"-"i" "e"-"j" "j"-"k" "k"-"l" "j"-"m" "m"-"n" "n"-"o" "n"-"p" "m"-"k"
}  \]

\begin{center}\tiny{
\begin{tabular}{ c| c c c c c c c c c c c c c c c c c c c c c c c c c }
 & $\phi(p)$ & $\tb{v}_{e_1}(0)$ & $\tb{v}_{e_2}(0)$ & $\tb{v}_{e_3}(0)$ & $\tb{v}_{e_4}(0)$  & $\tb{v}_{e_5}(0)$ & $\tb{v}_{e_6}(0)$ & $\tb{v}_{e_7}(0)$ & $\tb{v}_{e_8}(0)$ & $\tb{v}_{e_9}(0)$    \\
\cline{1-11}
 $v_1$ & $-n_{e_{10}}$ & 1 & 0 & 0 & 0 & 0 & 0 & 0 & 0 & 1 &    \\
 $v_2$ & 0 & -1 & 1 & 0 & 0 & 0 & 0 & 1 & 0 & 0 &     \\
 $v_3$ & 0 & 0 & -1 & 1 & 0 & 1 & 0 & 0 & 0 & 0 &   \\
 $v_4$ & 0 & 0 & 0 & -1 & -1 & 0 & 1 & 0 & 0 & 0 &    \\
 $v_5$ & 0 & 0 & 0 & 0 & 1 & -1 & 0 & 0 & 0 & 0 &   \\
 $v_6$ & 0 & 0 & 0 & 0 & 0 & 0 & -1 & 0 & 0 & 0 &    \\
 $v_7$ & 0 & 0 & 0 & 0 & 0 & 0 & 0 & -1 & 1 & 0 &   \\
 $v_8$ & 0 & 0 & 0 & 0 & 0 & 0 & 0 & 0 & -1 & 0 &    \\
 $v_9$ & 0 & 0 & 0 & 0 & 0 & 0 & 0 & 0 & 0 & -1 &   \\
 $l$ & 0 & 0 & 0 & 0 & 0 & -1 & 0 & 0 & 0 & 0 &     \\
\end{tabular}}
\end{center}
\vspace{1cm}
\caption{$B_G$ for the given $G$.  The vertices and edges are listed in the order in which they are traversed, starting from $v_1$.}\end{figure}

If $A_G$ has maximal rank, i.e. if rank$(A_G)=|\mm{Vert}(G)|+b_1$, then the solution space of this system will have dimension
\begin{eqnarray}\label{rankequality}
  |\mm{Vert}(G)|+1+|\mm{Edge}(G)|-(|\mm{Vert}(G)|+b_1) &=& |\mm{Edge}(G)|+1-b_1 \\
   &=& 3(b_1-1)+n+1-b_1 \nn\\
   &=& n-2+2b_1,\nn
\end{eqnarray}
where the second equality requires the trivalency of $G$.

Let $A_{ij}$ be the entry of $A$ that lies in the $i$th row and the $j$th column, and let $\tilde{A}_{ij}$ be the submatrix of
 $A$ given by eliminating the $i$th row and the $j$th column.  Then
$$\det A=\displaystyle\sum\limits_{j=1}^n (-1)^{i+j}A_{ij}\cdot\det(\tilde{A}_{ij})=\displaystyle\sum\limits_{i=1}^n (-1)^{i+j}
A_{ij}\cdot\det(\tilde{A}_{ij}).$$

Let $e$ be the marked edge of an arbitrary cycle.  As the Lagrangian condition $l_e$ is $n_e\phi(v)=\tb{v}_e(0)\mod\Lambda$, the corresponding row of $A_G$ is $l_e=(0\,\cdots\, n_e\,\cdots\, 0\,\cdots\, -1_e\,\cdots\,0)$, and considered as a row of $B_G$, takes the form $l_e=(0\,\cdots\, -1_e\,\cdots\, 0)$, where $1_e$ sits in the column corresponding to $\tb{v}_e(0)$.  Then $\det(B_G)=(-1)^{i_e+j_e}(-1_e)(\det (\tilde{B}_G)_{i_ej_e})$, where $-1_e$ sits in the $i_e$th row and $j_e$th column of $B_G$.  The final $b_1$ rows of the matrix $B_G$ all have this form, so

$$\det(B_G)=\pm\det C_G,$$
where $C_G$ is the matrix resulting from eliminating the $i_e$th row and $j_e$th column of each of the last $b_1$ rows of $B_G$.  The sign depends on the location of each $1_e$.

\begin{claim}\label{claim}
The matrix $C_G$ is upper triangular and is such that

\begin{equation*}
    (C_G)_{ii}=
    \begin{cases}
    -n_l&\text{if $i=1$}\\
    \pm1&\text{if $i\neq1$}
    \end{cases}
    \end{equation*}
where $l$ is the external leg to which the exceptional vertex $p$ is attached.
  \end{claim}

\begin{proof}(of claim \ref{claim}) Each edge $e$ is bounded by two vertices, so a column vector marked by $\tb{v}_e(0)$ will
 consist of at most 3 nonzero entries, the third coming from the edge $e$ possibly being marked.

The structure of these column vectors depends on when the edge is traversed in the sequence.  There are three possibilities:
\begin{itemize}
\item[i)] $e$ is not the first edge traversed after having backtracked and is not the final edge traversed in a loop
\item[ii)] $e$ is the first edge traversed after having backtracked
\item[iii)] $e$ is the final edge traversed in a minimal loop
\end{itemize}
These vectors take the forms
$$
\begin{array}{ccccc}
  \left(
  \begin{array}{c}
    * \\
    \pm1 \\
    \mp1 \\
    * \\
    * \\
    * \\
    * \\
  \end{array}
\right) & , & \left(
  \begin{array}{c}
    * \\
    \pm1 \\
    * \\
    \mp1 \\
    * \\
    * \\
    * \\
  \end{array}
\right) & , & \left(
  \begin{array}{c}
 * \\
    \pm1 \\
    * \\
    \mp1 \\
    * \\
    -1 \\
    * \\
  \end{array}
\right)
\end{array}
$$
respectively, where $*$ represents some number of 0 entries.  Indeed, if $e$ is of the first type, then the vertices bounding this edge give rise to successive rows in $B_G$.  If $e$ is the first traversed after having backtracked, then there will be several rows, depending on how many edges were retraced, in between the rows indexed by the vertices bounding $e$.  If $e$ is the final edge traversed in a loop consisting of $n$ edges, then there will be a gap of $n-1$ rows between the nonzero entries.  The $-1$ entry near the end comes from the fact that edges of the final type are exactly the marked edges.

Let $v_i$ be a vertex of $G$.  The nonzero entries of the row indexed by $v_i$ correspond to the edges which are connected to $v_i$.  The first nonzero entry corresponds either to the edge traversed just before $v_i$, or just after.  Because the first column is indexed by $\phi(p)$, the first nonzero entry of the row $v_i$ will lie at or after the $i^{\mm{th}}$ spot.  The matrix $C_G$ is therefore upper triangular.  Since $(B_G)_{11}=(C_G)_{11}=-n_l\neq0$, it remains only to show $(C_G)_{ii}\neq0$ for $2\leq i\leq|\mm{Vert}(G)|$.

Let $e_n$ be an edge of the first type listed above, bounded by $v_n$ and $v_{n+1}$.  Depending on the orientations of the edges connected to $v_{n+1}$, either the vector $\tb{v}_{e_n}(0)$ or $\tb{v}_{e_n}(1)$ contributes to the balancing condition at $v_{n+1}$, so $(B_G)_{(n+1)(n+1)}=\pm1$.

Now let $e_n$ be the first edge traversed after having backtracked, bounded by $v_i$ and $v_j$.  Assume $v_j$ is the one vertex of the two that has not yet been traversed.  The sequence of vertices and edges has the form
$$\cdots\longrightarrow e_{n-1}\longrightarrow v_{j-1}\longrightarrow v_i\longrightarrow e_n\longrightarrow v_j\longrightarrow \cdots,$$where $v_i$ was traversed before $v_{j-1}$, so does not index a new row.  Therefore, $(B_G)_{(j-1)(n)}$ and $(B_G)_{(j)(n+1)}$ are both nonzero.  In other words, if $i$ and $j$ are such that $(B_G)_{ij}$ is the second nonzero entry of the $j^{\mm{th}}$ column, and the $(j+1)$st column is indexed by an edge which is the first to be traversed after having backtracked, then $(B_G)_{(i+1)(j+1)}$ is nonzero as well.

Let $e_n$ be of the third type, and let $e_{n-1}$ be the second to last edge traversed in the given minimal loop.  The final entry of $\tb{v}_{e_{n-1}}(0)$ and the second entry of $\tb{v}_{e_n}(0)$ will lie in the same row, as $e_{n-1}$ and $e_n$ flow, in terms of the order in which they are traversed, through the same vertex.  This means that the second entry of $\tb{v}_{e_n}(0)$ has coordinates $(i,i+k)$, where $k$ is the number of minimal loops traversed, up to and including the one containing $e_n$.  Now, since $(B_G)_{ij}=(C_G)_{(i)(j-k)}$ for $1\leq i\leq|\mm{Vert}(G)|$, where $k$ is the number of minimal loops traversed before column $j$, the second entries of all of the column vectors of $B_G$ will slide into the diagonal upon taking cofactors at the nonzero entries of $B_G$ corresponding to the marked edges.

\end{proof}
Then $\det(B_G)=\pm n_l$ and since $n_l\neq0$ for all external legs $l$, the matrix $B_G$ is invertible, and the dimension of the solution space is $n-2+2b_1$.

Let $l$ be a leg of $G$ and recall that $\tb{v}_l(s)=\tb{v}_l(0)+n_l(\phi(s)-p_i)$ must point in the same direction as
$\phi'(s)$, where $p_i$ labels $l$ in $G$.  If $n_l<0$, then $n_l(\phi(s)-p_i)$ and $\tb{v}_l(s)$ point in opposite
directions, so  $|\tb{v}_l(0)|\geq |n_l(\phi(s)-p_i)|$ for all $s\in[0,1]$.  Since $l$ is a leg of $G$, $\tb{v}_l(0)=0$
by definition, so $\phi$ necessarily contracts $l$ to $p_i$.  The image under $\phi$ of the vertex bounding $l$ on its
opposite end is therefore determined.  If $n_l>0$ the image under $\phi$ of the vertex bounding $l$ on its opposite end is
unrestricted.  Because $\deg p_i=1$ if and only if $n_l<0$, and $\deg p_i=0$ if and only if $n_l>0$, the dimension
of $\dim G_{n,b,\s}^{\mm{trop}}$ is given by
\begin{eqnarray*}
  \dim G_{n,b,\s}^{\mm{trop}} &=& |\mm{Vert}(G)|+1+|\mm{Edge}(G)|-|\mm{Vert}(G)|-b_1-(\sum \deg p_i) \\
   &=& n-2+2b_1-(\sum \deg p_i)
\end{eqnarray*}
\begin{flushleft}
with the second equality coming from \eqref{rankequality}.
\end{flushleft}
\end{proof}



\subsection{Moduli Space Degeneration}\label{modulispacedegenerationsection}

In section \ref{ev} we define vectors $\overline{\mu}_{n,b}$, which will be shown to satisfy what we term the Quantum $A_{\infty}$ - relations, the genus zero case being the familiar $A_{\infty}$ - relations.  The vectors $\overline{\mu}_{n,b}$ will be defined in terms of moduli spaces of tropical Morse graphs, and the Quantum $A_{\infty}$- relations will be proven by examining degenerations of these moduli spaces in the case they are one-dimensional.

The moduli space is one-dimensional precisely when each interior point of the moduli space corresponds to a polygon or annulus with a single non-convex corner.  By \emph{non-convex corner} we mean a corner of a polygon $p$ for which we can find, for any epsilon disk $D$ centered at the corner, a chord $c$ for which $\partial c\subset \partial p$ and $(p\setminus \partial p)\,\cap \,(c\setminus \partial c)=\varnothing$.

The relations between the compositions $\overline{\mu}_{k,b}\circ_i\overline{\mu}_{l,b}$ are obtained by examining degenerations of 1-dimensional moduli spaces of tropical Morse trees and graphs.  These degenerations are best understood by considering the holomorphic polygons and annuli built from the trees and graphs via \ref{holopolygon}.  In \cite{Clay} the 1-dimensional moduli spaces come from the existence of a  non-convex corner in the given polygon.  In the case considered here the Lagrangians emanating from the non-convex corner cut into the interior of the polygon or annulus, and a degeneration occurs when one of the two Lagrangians hits the opposing side, splitting the polygon in two, or the annulus into a polygon.

\subsubsection{Topological Change Through a Vertex}

Consider a quadrivalent vertex of a stable graph and put a cyclic ordering on the half-edges.  Call two half-edges $h$, $h'$ \emph{adjacent} if either $h$ follows $h'$ or $h'$ follows $h$ in the cyclic ordering.  There are only two ways of pairing adjacent half-edges of a quadrivalent vertex.  The topology change in the domain of a 1-dimensional tropical Morse graph that occurs when the image of the free vertex moves past the image of a fixed vertex is determined by these two pairings.  Figure \ref{topologychangetriple} illustrates this.
\newpage
\vspace{1cm}
\begin{center}
\begin{picture}(0,0)
\put(-190,0){\line(1,-1){90}}
\put(-160,-30){\line(1,1){30}}
\put(-130,-60){\line(-1,-1){30}}
\put(-130,-60){\circle*{3}}
\put(-160,-30){\circle*{3}}
\put(-169,-37){$v$}
\put(-125,-59){$v_p$}
\put(-152,-52){$e$}
\put(-195,5){$h_1$}
\put(-130,5){$\ell_1$}
\put(-167,-100){$\ell_2$}
\put(-100,-100){$h_2$}
\put(-145,-35){\vector(1,-1){10}}

\put(-25,-20){\line(1,-1){50}}
\put(-25,-70){\line(1,1){50}}
\put(0,-45.5){\circle*{3}}
\put(-31,-15){$h_1$}
\put(25,-15){$\ell_1$}
\put(-31,-80){$\ell_2$}
\put(25,-81){$h_2$}
\put(5,-45){$_{\phi(v)=\phi(v_p)}$}

\put(100,0){\line(1,-1){90}}
\put(130,-30){\line(-1,-1){30}}
\put(160,-60){\line(1,1){30}}
\put(130,-30){\circle*{3}}
\put(160,-60){\circle*{3}}
\put(133,-27){$v_p$}
\put(150,-66){$v$}
\put(147,-45){$e'$}
\put(95,5){$h_1$}
\put(190,-100){$h_2$}
\put(94,-70){$\ell_2$}
\put(192,-28){$\ell_1$}
\put(134,-44){\vector(1,-1){10}}
\end{picture}
\end{center}

\vspace{1.5cm}
\begin{figure}[ht]
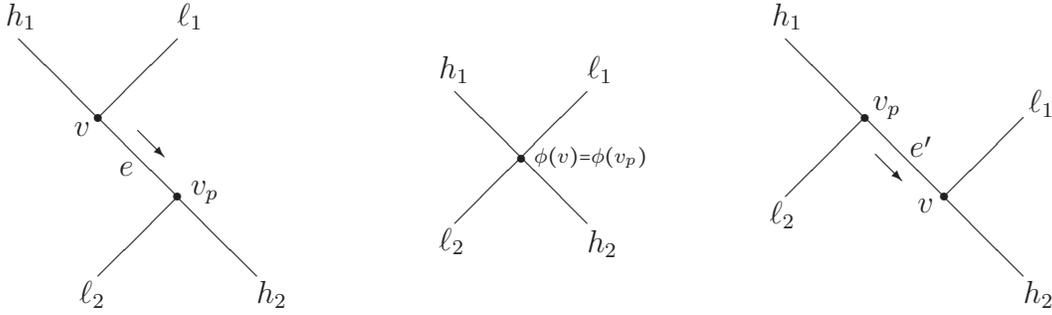

\[  \xygraph{
!{<0cm,0cm>;<1cm,0cm>:<0cm,1cm>::}
}  \]
\caption{As $e$ contracts, the half-edge $h_1$ is paired with $\ell_1$ and $h_2$ with $\ell_2$.  As the new edge $e'$ expands, the second of the two possible pairings is obtained i.e., $h_1$ with $\ell_2$ and $h_2$ with $\ell_1$.}
\label{topologychangetriple}
\end{figure}

Moving $\phi(v)$ toward $\phi(v_p)=p$ in $\ms{B}$ corresponds in the domain to contracting the edge $e$ that joins $v$ and $v_p$.  Similarly, as $\phi(v)$ moves away from $p$ in $\ms{B}$, the edge $e'$ extends in the domain.  The quadrivalent vertex in the domain, which corresponds to both $e$ and $e'$ being fully contracted, appears precisely when $\phi(v)=p$ in $\ms{B}$.

\subsubsection{The genus zero case}

A tropical Morse graph is said to \emph{degenerate} when the length of one of its edges goes to infinity.  Recall part (6) of Definition \ref{TMGDef} defines the length of an edge $e$ of a tropical Morse graph as $\frac{1}{n_e}\mm{log}\left(\frac{\tb{v}_e(1)}{\tb{v}_e(0)}\right)$.
The genus zero case of the degeneration of a one-dimensional moduli space of curves has been thoroughly discussed in \cite{Clay} and various other sources.  The degeneration is realized topologically as a ``bubbling off" of a holomorphic disk from another holomorphic disk, and is described here pictorially as the splitting of a holomorphic polygon in $T\ms{B}$ into two other holomorphic polygons.

For the sake of connecting our intuition to the definition, consider the point of $G^{\mm{trop}}_{4,0}$ on the right-hand side of Figure \ref{ThreePntsOfModuliSpace}.  Note that in this example, the edge bounded externally by $p_{0,3}$ is oriented outward.  Starting counterclockwise from the edge bounded externally by $p_{0,1}$, label the external edges as $e_1,e_2,e_4,e_5$, and label the internal edge $e_3$.  Intuitively, this TMG degenerates when the given triangle in $X(\ms{B})$ splits into two triangles.  This occurs precisely when $|\tb{v}_{e_3}(1)|=0$.  The balancing condition at $v$ implies $\phi(v)=\frac{1}{n_1-n_0}\tb{v}_{e_5}(0)+p_{0,1}\mod\Lambda$.  Because $n_{e_3}<0$, we have
$|e_3| = \frac{1}{n_{e_3}}\mm{log}\left(\frac{\tb{v}_{e_3}(1)}{\tb{v}_{e_3}(0)}\right)\to \infty$ as $\tb{v}_{e_3}(1)\to 0$, so our intuition agrees with the definition of degeneration.

The degenerations come in pairs and result from the existence of a non-convex corner in each point of the interior of the moduli space.

This is illustrated in Figure \ref{ThreePntsOfModuliSpace}.

\newpage

\begin{center}

\begin{picture}(0,0)


\put(-200,-100){\line(1,1){50}}
\put(-200,-100){\line(3,1){55}}
\put(-150,-50){\line(2,-1){63}}
\put(-87,-81.55){\line(-1,0){82}}

\put(-170,-75){$_{\phi(v)}$}
\put(-120,-95){$L_0$}
\put(-120,-60){$L_1$}
\put(-190,-70){$L_2$}
\put(-175,-105){$L_3$}

\put(-210,-105){$_{p_{2,3}}$}
\put(-160,-42){$_{p_{1,2}}$}
\put(-85,-80){$_{p_{0,1}}$}
\put(-150,-90){$_{p_{0,3}}$}

\put(-200,-150){\line(1,0){110}}
\put(-125,-160){\line(-1,0){50}} 
\put(-220,-150){$_{p_{2,3}}$}
\put(-120,-160){$_{p_{0,3}}$}
\put(-85,-150){$_{p_{0,1}}$}
\put(-185,-140){$_{\phi(v)}$}
\put(-150,-140){$_{p_{1,2}}$}
\put(-200,-150){\circle*{3}}
\put(-175,-150){\circle*{3}}
\put(-145,-150){\circle*{3}}
\put(-90,-150){\circle*{3}} 
\put(-175,-160){\circle*{3}} 
\put(-125,-160){\circle*{3}} 

\put(-200,-150){\vector(1,0){15}}
\put(-150,-150){\vector(-1,0){13}}
\put(-90,-150){\vector(-1,0){30}}
\put(-175,-160){\vector(1,0){30}}

\put(-45,-100){\line(1,1){50}}
\put(-45,-100){\line(3,1){55}}
\put(5,-50){\line(2,-1){63}}
\put(68,-81.55){\line(-1,0){68}}

\put(35,-95){$L_0$}
\put(35,-60){$L_1$}
\put(-35,-70){$L_2$}
\put(-20,-105){$L_3$}

\put(-55,-105){$_{p_{2,3}}$}
\put(-5,-42){$_{p_{1,2}}$}
\put(70,-80){$_{p_{0,1}}$}
\put(-15,-75){$_{p_{1,2}=\phi(v)}$}
\put(10,-90){$_{p_{0,3}}$}

\put(-35,-150){\line(1,0){95}}
\put(-55,-150){$_{p_{2,3}}$}
\put(65,-150){$_{p_{0,1}}$}
\put(-15,-140){$_{p_{1,2}=\phi(v)}$}
\put(20,-160){$_{p_{0,3}}$}

\put(-35,-150){\circle*{3}}
\put(60,-150){\circle*{3}} 
\put(0,-150){\circle*{3}} 
\put(25,-150){\circle*{3}} 

\put(-35,-150){\vector(1,0){20}}
\put(0,-150){\vector(1,0){15}}
\put(60,-150){\vector(-1,0){20}}

\put(110,-100){\line(1,1){50}}
\put(110,-100){\line(3,1){75}}
\put(160,-50){\line(2,-1){63}}
\put(223,-81.55){\line(-1,0){58}}

\put(175,-70){$_{\phi(v)}$}

\put(190,-95){$L_0$}
\put(190,-60){$L_1$}
\put(120,-70){$L_2$}
\put(135,-105){$L_3$}

\put(100,-105){$_{p_{2,3}}$}
\put(150,-42){$_{p_{1,2}}$}
\put(225,-80){$_{p_{0,1}}$}
\put(160,-90){$_{p_{0,3}}$}

\put(115,-150){\line(1,0){105}}
\put(195,-160){\line(-1,0){25}}
\put(95,-150){$_{p_{2,3}}$}
\put(150,-160){$_{p_{0,3}}$}
\put(225,-150){$_{p_{0,1}}$}
\put(195,-140){$_{\phi(v)}$}
\put(150,-140){$_{p_{1,2}}$}
\put(115,-150){\circle*{3}}
\put(195,-150){\circle*{3}}
\put(155,-150){\circle*{3}}
\put(220,-150){\circle*{3}} 
\put(195,-160){\circle*{3}} 
\put(170,-160){\circle*{3}} 

\put(115,-150){\vector(1,0){20}}
\put(150,-150){\vector(1,0){30}}
\put(215,-150){\vector(-1,0){10}}
\put(195,-160){\vector(-1,0){15}}


\put(-110,-205){\line(-1,-1){40}}
\put(-150,-205){\line(1,-1){20}}
\put(-170,-225){\line(1,-1){20}}
\put(-150,-245){\line(0,-1){30}}

\put(-115,-200){$_{p_{0,1}}$}
\put(-155,-200){$_{p_{1,2}}$}
\put(-180,-220){$_{p_{2,3}}$}
\put(-155,-280){$_{p_{0,3}}$}

\put(-115,-220){$+$}
\put(-150,-220){$-$}
\put(-170,-240){$+$}
\put(-160,-265){$-$}

\put(-145,-250){$v$}


\put(-10,-220){\line(1,-1){25}}
\put(40,-220){\line(-1,-1){25}}
\put(15,-245){\line(0,-1){30}}
\put(15,-220){\line(0,-1){25}}

\put(5,-215){$_{p_{1,2}}$}
\put(-18,-215){$_{p_{2,3}}$}
\put(8,-280){$_{p_{0,3}}$}
\put(33,-215){$_{p_{0,1}}$}

\put(-10,-235){$+$}
\put(5,-232){$-$}
\put(33,-235){$+$}
\put(5,-265){$-$}

\put(18,-250){$v$}

\put(130,-205){\line(1,-1){40}}
\put(170,-205){\line(-1,-1){20}}
\put(190,-225){\line(-1,-1){20}}
\put(170,-245){\line(0,-1){30}}

\put(125,-200){$_{p_{2,3}}$}
\put(165,-200){$_{p_{1,2}}$}
\put(190,-220){$_{p_{0,1}}$}
\put(165,-280){$_{p_{0,3}}$}

\put(130,-220){$+$}
\put(165,-220){$-$}
\put(185,-240){$+$}
\put(160,-265){$-$}

\put(175,-250){$v$}

\end{picture}


\end{center}

\vspace{8.5cm}
\begin{figure}[ht]
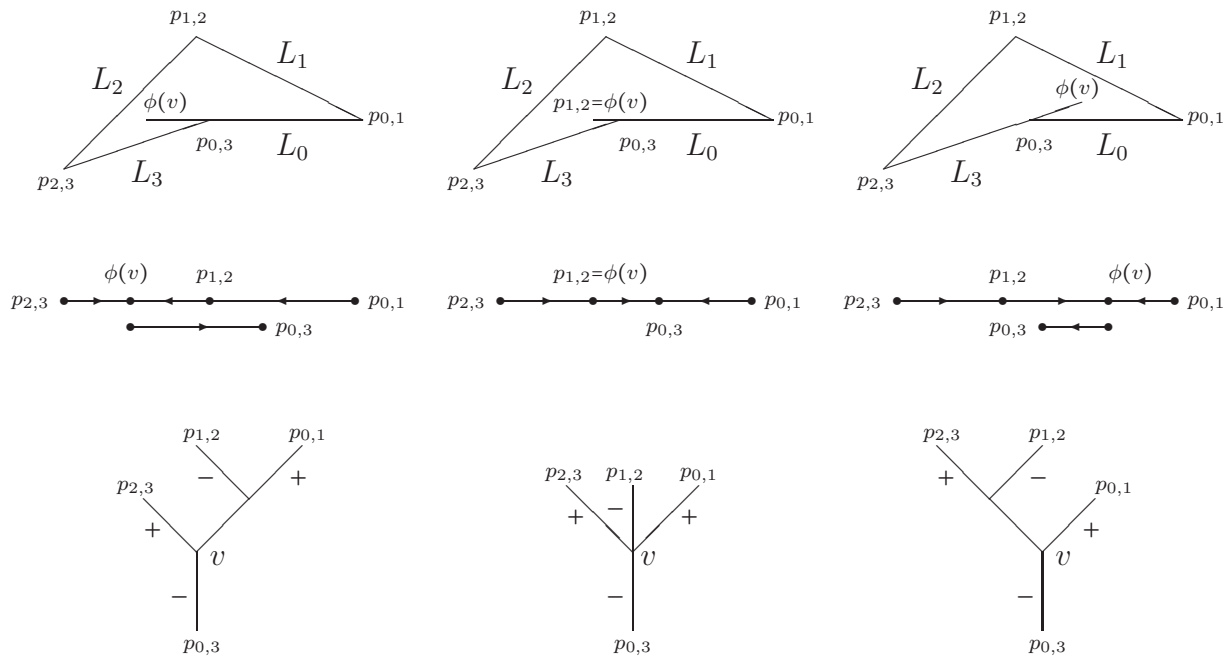

\[  \xygraph{
!{<0cm,0cm>;<1cm,0cm>:<0cm,1cm>::}
}  \]
\caption{Three sample points in the one-dimensional moduli space $G_{4,0}^{\mm{trop}}$.  The bottom row consists of the three domains $\{T_1,T_2,T_3\}$.  The middle row consists of the images $\{\phi(T_i)\}$, and the top row consists of the polygons built from the trees as viewed in $TB$.  The free vertex is labeled by $v$ and the left degeneration occurs when the polygon $(L_0,L_1,L_2,L_3)$ splits into $(L_0,L_1,L_2)$ and $(L_0,L_2,L_3)$.  Similarly for the right degeneration.  Notice the topological change which occurs in the domain, as explained in Figure \ref{topologychangetriple}.}
\label{ThreePntsOfModuliSpace}
\end{figure}

\newpage

\subsubsection{The Genus One Case}

The genus one case is similar to that of the genus zero case in that each degeneration is the result of a Lagrangian, emanating from a non-convex corner, which eventually intersects an opposing edge of the figure built from the given tropical Morse graph, creating two additional corners.  The degeneration on the other end of the moduli space arises similarly from the other Lagrangian which comprises the non-convex corner.  The difference is that instead of a polygon degenerating into two polygons in two different ways, an annulus degenerates into a polygon in two different ways.  This is illustrated in Figure \ref{ThreePntsOfModuliSpaceGenusOne}.


\begin{center}

\vspace{-2cm}
\begin{picture}(0,0)

\put(-200,-100){\line(0,-1){45}}
\put(-200,-145){\line(1,0){110}}
\put(-90,-100){\line(0,-1){45}}
\put(-200,-100){\line(2,-1){55}}
\put(-90,-100){\line(-2,-1){80}}

\put(-210,-95){$_{q_{1,0}}$}
\put(-95,-95){$_{q_{1,0}}$}
\put(-150,-115){$_{p_{0,1}}$}
\put(-140,-155){$_{p_{2,2}}$}
\put(-135,-145){\circle*{3}}
\put(-190,-135){$_{\phi(v)}$}

\put(-180,-105){$L_1$}
\put(-125,-105){$L_0$}
\put(-170,-160){$L_2$}

\put(-200,-195){\line(1,0){110}}
\put(-155,-205){\line(-1,0){25}}
\put(-100,-195){\vector(-1,0){15}}
\put(-130,-195){\vector(-1,0){30}}
\put(-130,-195){\vector(-1,0){62}}
\put(-160,-205){\vector(-1,0){10}}

\put(-210,-185){$_{q_{1,0}}$}
\put(-95,-185){$_{q_{1,0}}$}
\put(-150,-205){$_{p_{0,1}}$}
\put(-140,-185){$_{p_{2,2}}$}
\put(-185,-185){$_{\phi(v)}$}

\put(-200,-195){\circle*{3}}
\put(-90,-195){\circle*{3}}
\put(-180,-205){\circle*{3}} %
\put(-135,-195){\circle*{3}}
\put(-180,-195){\circle*{3}}  %
\put(-155,-205){\circle*{3}}

\put(-145,-270){\circle{180}}
\put(-145,-250){\line(0,-1){20}}
\put(-165,-270){\line(-1,0){20}}
\put(-125,-270){\line(1,0){20}}
\put(-152,-275){$_{p_{2,2}}$}
\put(-200,-270){$_{p_{0,1}}$}
\put(-100,-270){$_{q_{1,0}}$}

\put(-180,-268){$+$}
\put(-120,-268){$-$}
\put(-154,-264){$_0$}
\put(-175,-280){$v$}

\put(-40,-100){\line(0,-1){45}}
\put(-40,-145){\line(1,0){110}}
\put(70,-100){\line(0,-1){45}}
\put(-40,-100){\line(2,-1){65}}
\put(70,-100){\line(-2,-1){55}}

\put(-50,-95){$_{q_{1,0}}$}
\put(65,-95){$_{q_{1,0}}$}
\put(10,-115){$_{p_{0,1}}$}
\put(20,-155){$_{p_{2,2}}$}
\put(25,-145){\circle*{3}}
\put(27,-130){$_{\phi(v)=p_{2,2}}$}

\put(-20,-105){$L_1$}
\put(35,-105){$L_0$}
\put(-10,-160){$L_2$}

\put(-40,-195){\line(1,0){110}}
\put(10,-205){\line(1,0){15}}
\put(70,-195){\vector(-1,0){25}}
\put(70,-195){\vector(-1,0){80}}
\put(10,-205){\vector(1,0){10}}

\put(-50,-185){$_{q_{1,0}}$}
\put(65,-185){$_{q_{1,0}}$}
\put(-10,-205){$_{p_{0,1}}$}
\put(8,-185){$_{p_{2,2}=\phi(v)}$}

\put(-40,-195){\circle*{3}}
\put(70,-195){\circle*{3}}
\put(25,-195){\circle*{3}}
\put(10,-205){\circle*{3}}
\put(25,-205){\circle*{3}}

\put(15,-270){\circle{180}}
\put(-5,-270){\line(1,0){20}}
\put(-5,-270){\line(-1,0){20}}
\put(35,-270){\line(1,0){20}}
\put(8,-275){$_{p_{2,2}}$}
\put(-40,-270){$_{p_{0,1}}$}
\put(60,-270){$_{q_{1,0}}$}

\put(-20,-268){$+$}
\put(40,-268){$-$}
\put(6,-264){$_0$}
\put(-15,-280){$v$}

\put(120,-100){\line(0,-1){45}}
\put(120,-145){\line(1,0){110}}
\put(230,-100){\line(0,-1){45}}
\put(120,-100){\line(2,-1){80}}
\put(230,-100){\line(-2,-1){55}}

\put(110,-95){$_{q_{1,0}}$}
\put(220,-95){$_{q_{1,0}}$}
\put(170,-115){$_{p_{0,1}}$}
\put(180,-155){$_{p_{2,2}}$}
\put(185,-145){\circle*{3}}
\put(200,-135){$_{\phi(v)}$}

\put(140,-105){$L_1$}
\put(195,-105){$L_0$}
\put(150,-160){$L_2$}

\put(120,-195){\line(1,0){110}}
\put(175,-205){\line(1,0){25}}
\put(230,-195){\vector(-1,0){20}}
\put(230,-195){\vector(-1,0){40}}
\put(230,-195){\vector(-1,0){80}}
\put(175,-205){\vector(1,0){15}}

\put(110,-185){$_{q_{1,0}}$}
\put(225,-185){$_{q_{1,0}}$}
\put(155,-205){$_{p_{0,1}}$}
\put(180,-185){$_{p_{2,2}}$}
\put(205,-205){$_{\phi(v)}$}

\put(120,-195){\circle*{3}}
\put(230,-195){\circle*{3}}
\put(200,-205){\circle*{3}}
\put(175,-205){\circle*{3}}
\put(185,-195){\circle*{3}}
\put(200,-195){\circle*{3}}

\put(175,-270){\circle{180}}
\put(175,-290){\line(0,1){20}}
\put(155,-270){\line(-1,0){20}}
\put(195,-270){\line(1,0){20}}
\put(170,-265){$_{p_{2,2}}$}
\put(120,-270){$_{p_{0,1}}$}
\put(220,-270){$_{q_{1,0}}$}

\put(140,-268){$+$}
\put(200,-268){$-$}
\put(167,-277){$_0$}
\put(145,-280){$v$}

\end{picture}\end{center}

\vspace{8cm}
\begin{figure}[ht]
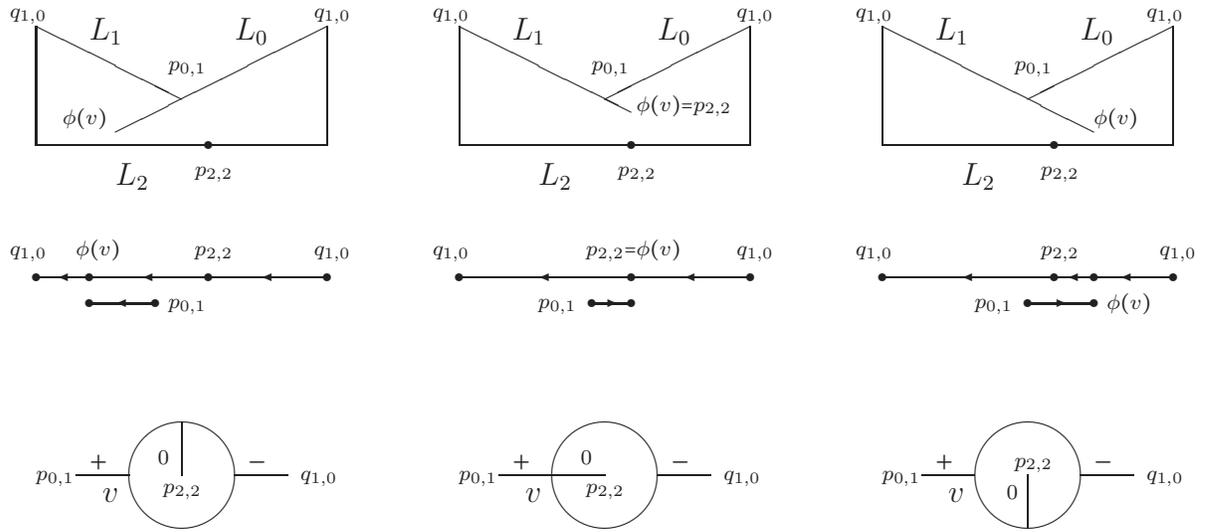

\renewcommand{\figurename}{Figure}
\[  \xygraph{
!{<0cm,0cm>;<1cm,0cm>:<0cm,1cm>::}
}  \]
\caption{Three sample points in the one-dimensional moduli space $G_{3,1}^{\mm{trop}}$.  The free vertex is labeled by $v$ and the left degeneration occurs when the annulus $\partial u=L_0\cup (L_1,L_2)$ breaks into the polygon $(L_0,L_1,L_2)$.}
\label{ThreePntsOfModuliSpaceGenusOne}
\end{figure}



\begin{lem}\label{newdeg1}
The degeneration of an annulus $u:G\longrightarrow B$ with $b_1(G)=1$ creates exactly one additional degree 1 corner and one additional degree 0 corner.
\end{lem}
\begin{proof}
Let $u:G\longrightarrow \ms{B}$ be such an annulus and let $\overline{u}:\overline{G}\longrightarrow \ms{B}$ be the polygon to which it degenerates.  We may write the set of corners of $\overline{u}$ as $\{p_i\}_{i=1}^n\sqcup\{q_1,q_2\}$, where $\{p_i\}_{i=1}^n$ is the set of corners of $u$.  Because $\dim G=1$ and $\dim \overline{G}=0$, Theorem \ref{DimensionTheorem} implies
\begin{equation}\label{dim1annulus}
  \sum\deg p_i = n-2 + 2b_1(G)-1
\end{equation}
and
\begin{equation}\label{dim0polygon}
  \sum\deg p_i+\deg q_1+\deg q_2 = (n+2)-2+2b_1(\overline{G})
\end{equation}
Substituting \eqref{dim1annulus} into \eqref{dim0polygon} yields $\deg q_1+\deg q_2 = 1$.  The result follows from the fact that the degree function is nonnegative by definition.
\end{proof}

\begin{lem}\label{main}
If $u,v$ are the two degenerate annuli that bound the one-dimensional moduli space $G_{n,1}^{\mm{trop}}$, then $\deg u=\deg v$ and $s(u)=s(v)$.
\end{lem}
\begin{proof}
Denote by $u_{\phi(w)}$ the image of an arbitrary point of $G_{n,1}^{\mm{trop}}$, where $w\in\mm{Vert}(G)$ is the unique vertex whose image is not determined by the data.  Let $L$ be either of the Lagrangians emanating from the non-convex corner of $u_{\phi(w)}$.  Because $\phi(w)$ extends $L$ from the non-convex vertex of $u_{\phi(w)}$ into the interior of $u_{\phi(w)}$, the area enclosed by $u_{\phi(w)}$ remains constant (see Figure \ref{ThreePntsOfModuliSpaceGenusOne}).  This gives $\deg u=\deg v$.

By Lemma \ref{newdeg1}, image of each degenerate genus 1 graph has one additional degree-zero corner, which does not contribute to the overall sign of the polygon, and one additional degree-one corner.  Let $p,q\in T\ms{B}$ be these new degree-one corners corresponding to a given degenerate pair.  Given either of the images of the two degenerations, the newly created degree-one corner corresponds to the meeting place of $i)$ one of the Lagrangians containing the non-convex corner giving rise to the degeneration, and $ii)$ the Lagrangian containing the boundary component being split, which is the same for both degenerations.  Both $p$ and $q$ therefore lie on the same Lagrangian, which we denote by $L_c$.  The signs of $p$ and $q$ are determined by the agreement or disagreement of the left-to-right orientation of $L_c$ with the orientation induced on $L_c$ by the natural orientation on $\mb{C}$.  In a counterclockwise traversal of the boundaries of the images of the respective degenerate annuli, $L_c$ will be encountered immediately after passing through both $p$ and $q$, yielding the same sign for both $p$ and $q$.  This gives $s(u)=s(v)$.

\end{proof}

\begin{lem}\label{cyclelemma}
Assume $\dim G_{n,1}^{\mm{trop}}=1$ and let $\phi:G\longrightarrow \ms{B}$ be a point of $G_{n,1}^{\mm{trop}}$.  If $v\in\mm{Vert}(G)$ is the unique vertex whose image is not determined by the initial data, then $v$ must belong to the unique cycle of $G$.
\end{lem}
\begin{proof}
The tropical Morse graph $G$ is trivalent and $n-1$ of the external legs are degree 1, so if $v$ does not belong to the cycle, then at least one of the flags to which $v$ is attached is either a degree 1 external leg, or comprises part of an internal edge that forms the base of a contractible subgraph, all of whose external legs are degree 1.  If two adjacent external legs are degree one, not only are the two legs contracted by $\phi$, but the internal edge to which these legs are attached must be contracted as well.  Indeed, if $n_e\leq0$ and $\tb{v}_e(0)=0$, then $\phi|_e=\text{constant}$, and if $n_e\geq0$ and $\tb{v}(1)=0$, then $\phi|_e=\text{constant}$.  In either case $v$ is fixed by $\phi$, because $\phi$ is constant on the aforementioned contractible subgraph.

\end{proof}

\section{The Fukaya Category}

Recall the definition (\ref{DefOfBZ}) of $\ms{B}(\frac{1}{n}\mb{Z})$.  The spaces $\ms{B}(\frac{1}{n}\mb{Z})$ and $\ms{B}(\frac{1}{-n}\mb{Z})$ are dual by the pairing \begin{equation*}
    <p,q>=
    \begin{cases}
    1&\text{if $p=q$}\\
    0&\text{if $p\neq q$}
    \end{cases}
    \end{equation*}
It is this duality which allows us to construct a modular operad similar to $\mc{E}_V$ (Section \ref{evsings}) out of the Hom-spaces of the Fukaya Category.

\begin{defn}  The objects of the precategory Fuk$(X(\mb{R}/d\mb{Z}))$ are the Lagrangian submanifolds of $X(\mb{R}/d\mb{Z})$ defined by $\s_n(\mb{R}/d\mb{Z})$.  Note that since $X(\mb{R}/d\mb{Z})$ is an elliptic curve, any 1-dimensional submanifold is Lagrangian.
The Hom-spaces of this category, $\mm{Hom}(L_{n_i},L_{n_j})$, are given by
\begin{equation*}
    \mm{Hom}(L_{n_i},L_{n_j})=
    \begin{cases}
    \bigoplus_{p\in \ms{B}(\frac{1}{n_j-n_i}\mb{Z})}[p]\Lambda_{\mm{nov}}\longrightarrow 0&\text{if $n_i<n_j$}\\
    0\longrightarrow\bigoplus_{p\in \ms{B}(\frac{1}{n_j-n_i}\mb{Z})}[p]\Lambda_{\mm{nov}}&\text{if $n_i>n_j$},
    \end{cases}
    \end{equation*}
where the first column is degree 0, the second is degree 1, and $\Lambda_{\mm{nov}}$ is the Novikov ring, i.e., the ring of power series $\sum_{i=1}^{\infty}a_iq^{\lambda_i}$ where $a_i\in\mb{C}$, $\lambda_i\in\mb{R}_{\geq0}$, $\lambda_i\to\infty$ as $i\to\infty$.  A typical element of $\mm{Hom}(L_{n_i},L_{n_j})$ for $\ms{B}(\frac{1}{n_j-n_i}\mb{Z})=\{p_1,\dots,p_n\}$ has the form $$\sum\limits_{i\in\mb{Z}}(a_{p_1})_iq^{(\lambda_{p_1})_i}[p_1]+\cdots+ \sum\limits_{i\in\mb{Z}}(a_{p_n})_iq^{(\lambda_{p_n})_i}[p_n],$$and notice it makes perfect sense to write $p\in\mm{Hom}(L_i,L_j)$ for $p\in L_i\cap L_j$.  The exponents $\lambda$ will be seen in Section \ref{ModOPSection} to be the areas of the polygons comprising the zero-dimensional moduli spaces $G_{n,b,\s}^{\mm{trop}}(p_{1,2},p_{2,3},\dots,p_{n,1})$.  In what follows the numbers $\lambda_{p_i}$ will correspond to the area of a certain polygon in $T\ms{B}$. 

\end{defn}

\subsection{The Perturbation}\label{perturbationsection}

The objects of the new category are pairs $(L,\tb{v})$, where $L$ is a Lagrangian given as before, $\tb{v}$ is a constant section of $T\ms{B}$, and $(L,\tb{v})$ is taken to mean the Lagrangian $L+\tb{v}$ in $T\ms{B}$.  The Hom-space $\mm{Hom}((L_{n_i},\tb{v}),(L_{n_j},\tb{w}))$ is generated over $\Lambda_{\mm{nov}}$ by the points of $\ms{B}(\frac{1}{n_j-n_i}\mb{Z})$, and sits in degree 0 if $n_j>n_i$ and degree 1 if $n_i>n_j$.  If $n_i=n_j$ then
$\mm{Hom}((L_{n_i},\tb{v}),(L_{n_j},\tb{w}))$ is given by the self-dual chain complex
$$\cdots\longrightarrow0\longrightarrow\mb{C}[p]\longrightarrow 0\longrightarrow\cdots$$
where $p$ is an arbitrarily chosen point of $\ms{B}$.

A sequence $\{(L_{n_1},\tb{v}_1),\dots,(L_{n_k},\tb{v}_k)\}$ is transversal if $\tb{v}_1<\cdots<\tb{v}_k$, and
\begin{eqnarray*}
   && \mm{max}\{||\tb{v}_{i+1}-\tb{v}_i|-|\tb{v}_{j+1}-\tb{v}_j||:1\leq i<j\leq k\} \\
   &<& \frac{1}{2}\mm{min}\{d(p,q):p\neq q\,\,\,\mm{and}\,\,\,p,q\in\bigcup_{\substack{i,j\\i\neq j}}\pi((L_{n_i}+\tb{v}_i)\cap (L_{n_j}+\tb{v}_j))\},
\end{eqnarray*}
where $d(p,q)$ is the length of the shorter of the two arcs created by partitioning $\ms{B}$ with $p$ and $q$.

Every point of $\pi((L_{n_i}+\tb{v})\cap (L_{n_j}+\tb{w}))$ can be written $p_{i,j}+(\tb{w}-\tb{v})$ for $p_{i,j}\in \ms{B}(\frac{1}{n_j-n_i}\mb{Z})=\pi(L_{n_i}\cap L_{n_j})$.  The space $G_{k,b}^{\mm{trop}}(p_1+(\tb{v}_2-\tb{v}_1),\dots,p_k+(\tb{v}_{k+1}-\tb{v}_k))$ is defined as usual if $n_i\neq n_j$ for all adjacent Lagrangians $L_{n_i}$ and $L_{n_j}$ in a given cycle.  If $v$ is an external vertex of a graph bounding a leg labeled by $n_{i+1}-n_i=0$, then the image of $v$ under $\phi$ can lie anywhere on $\ms{B}$ if $\deg p+(\tb{v}_{i+1}-\tb{v}_i)=0$, and $\phi(v)=p+(\tb{v}_{i+1}-\tb{v}_i)$ if $\deg p+(\tb{v}_{i+1}-\tb{v}_i)=1$.  Any external leg labeled by $n_e=0$ is contracted.

The transversality condition ensures that the points $p_i+(\tb{v}_{i+1}-\tb{v}_i)$ defining the moduli spaces $G_{k,b}^{\mm{trop}}(p_1+(\tb{v}_2-\tb{v}_1),\dots,p_k+(\tb{v}_{k+1}-\tb{v}_k))$ are always distinct, which ensures the actual and expected dimensions of $G_{k-1,b}^{\mm{trop}}$ are equal.  With these redefined moduli spaces, the category Fuk($X(\ms{B})$) is now a genuine $A_{\infty}$-precategory.  See \cite{Clay} for an exposition of $A_{\infty}$-precategories.

\begin{lem}
The transversality condition ensures that the points $p_i+(\tb{v}_{i+1}-\tb{v}_i)$ defining the moduli spaces
 $G_{k,b}^{\mm{trop}}(p_1+(\tb{v}_2-\tb{v}_1),\dots,p_k+(\tb{v}_{k+1}-\tb{v}_k))$ are always distinct.
\end{lem}
\begin{proof}
Consider two points $p_i\in \pi((L_{n_i}+\tb{v}_i)\cap (L_{n_{i+1}}+\tb{v}_{i+1}))$ and
$p_j\in \pi((L_{n_j}+\tb{v}_j)\cap (L_{n_{j+1}}+\tb{v}_{j+1}))$ for $i<j$.  The result follows
immediately if $p_i=p_j$, so assume $p_i\neq p_j$.

Write $p_i'$ for $p_i+(\tb{v}_{i+1}-\tb{v}_i)$.  We then have
\begin{eqnarray*}
  |p_i'-p_j'| &=& |p_i-p_j+(\tb{v}_{i+1}-\tb{v}_i)-(\tb{v}_{j+1}-\tb{v}_j)| \\
   &=& |p_i-p_j - [(\tb{v}_{j+1}-\tb{v}_j)-(\tb{v}_{i+1}-\tb{v}_i)])| \\
   &\geq& ||p_i-p_j|-|(\tb{v}_{i+1}-\tb{v}_i)-(\tb{v}_{j+1}-\tb{v}_j)|| \\
   &=& ||p_i-p_j|-||\tb{v}_{i+1}-\tb{v}_i|-|\tb{v}_{j+1}-\tb{v}_j||| \\
   &\geq& |p_i-p_j| - ||\tb{v}_{i+1}-\tb{v}_i|-|\tb{v}_{j+1}-\tb{v}_j|| \\
   &\geq& \underset{i,j}{\mm{min}} \{|p_i-p_j|\} - \frac{1}{2}\underset{i,j}{\mm{min}}\{|p_i-p_j|\}  \\
   &=& \underset{i,j}{\mm{min}}\{|p_i-p_j|\} \\
   &>& 0,
\end{eqnarray*}
where the final inequality holds by assumption.
\end{proof}



\section{The Quantum $A_{\infty}$-Relations}

We define the Quantum $A_{\infty}$-relations in terms of an algebra structure over the Feynman Transform of a modular operad, and prove these relations are a generalization of the usual $A_{\infty}$-relations.

\subsection{The Modular Operad $\tilde{\mb{S}}[t]$}

\subsubsection{Motivation}

Recall that we associate a Riemman surface with boundary to each point of the moduli space $G_{n,b,\s}^{\mm{trop}}$.  The operad $\St$ gives a template for algebraically encoding certain aspects of these moduli spaces and their degenerations.   Each closed boundary component is represented as a cycle in a symmetric group $\mb{S}_k$, and  degenerations are represented, depending on the genus, as either the merging or severing of cycles, as described in the definition of the composition $\mu^{\St}$.

\subsubsection{Definition of $\St$}
In \cite{Bar}, the operad $\tilde{\mb{S}}[t]$ was defined homologically with the underlying $\mb{S}$-module $k[S_n]\otimes k[t]$, and $$\tilde{\mb{S}}[t]((n,b))=\bigoplus\limits_{\substack{\s\\b=2g+i_{\s}-1}}k\cdot\s t^g$$ where $\deg\s t^g=-b$ and $i_{\s}$ is the number of disjoint cycles into which $\s$ can be factored.  This operad is redefined cohomologically as follows.

The underlying $\mb{S}$-module remains the same, but $\deg \s t^g$ changes from $-b$ to $n-(-b)$.  By Proposition 2.23 of \cite{Mod}, all compositions $\mu_G^{\St}:D(G)\otimes\St((G))\longrightarrow \St((n,b))$ can be built from compositions parameterized by contracting either the single edge of $G_0=*_{n_1,b_1}\coprod *_{n_2,b_2}$, or the single edge of $G_1=*^1_{n,b}$.  Respectively, these compositions take the forms

\begin{equation}\label{Comp}
    \begin{diagram}
    \node{D(G_0)\otimes\St((n_1,b_1))\otimes\St((n_2,b_2))}\arrow{e,t}{}\node{\St((n_1+n_2-2,b_1+b_2))}
    \end{diagram}
\end{equation}
$$\alpha\otimes(\s t^g\otimes \tau t^{g'})\mapsto \pi_{ff'}(\s\tau(ff'))t^{g+g'}$$
and
\begin{equation}\label{Compo}
    \begin{diagram}
    \node{D(G_1)\otimes\St((n,b))}\arrow{e,t}{}\node{\St((n-2,b+1))}
    \end{diagram}
\end{equation}
$$\beta\otimes \s t^g\mapsto \pi_{ff'}(\s(ff'))t^{g+1}$$
where $\pi_{ff'}$ is defined by
$$\pi_{ff'}:\mm{Aut}(\{1,\dots,n\}\sqcup\{f,f'\})\longrightarrow \mb{S}_n$$
$$(i_1\cdots \,i_{k}fj_1\cdots j_{\ell}f')\mapsto (i_1\cdots \,i_{k}j_1\cdots j_{\ell})$$
and $\alpha$ and $\beta$ are oriented twists, derived below.

In the first case $\St((n_1,b_1))\otimes\St((n_2,b_2))$ sits in degree $(n_1+n_2)+(b_2+b_1)$ and
$\St((n_1+n_2-2,b_1+b_2))$ sits in degree $n_1+n_2-2+b_1+b_2$.  The requirement that $\deg\mu_{G_0}^{\St}=0$ means that the source should be shifted two places to the left.  For this purpose, we twist by
 \begin{eqnarray*}\label{MappingG0}
   \mc{K}^{-2}\mathcal{B}(G_0) &\simeq& \mb{C}[-2]\otimes\mb{C}[b_1=0] \\
    &\simeq& \mb{C}[-2],
 \end{eqnarray*}
where the cocycles $\mc{K}$ and $\mc{B}$ were defined in Example \ref{CocycleExample}.

In the second case, $\St((n,b))$ should be shifted one place to the left.  Again, we twist by
 \begin{eqnarray*}\label{MappingG1}
   \mc{K}^{-2}\mathcal{B}(G_1) &\simeq& \mb{C}[-2]\otimes\mb{C}[b_1=1] \\
    &\simeq& \mb{C}[-1]
 \end{eqnarray*}

The maps \eqref{Comp} and \eqref{Compo} do not produce any signs, so $\St$ is in fact a modular $D=\mc{K}^{-2}\mathcal{B}$-operad.

\subsection{Construction of the Quantum $A_{\infty}$-Relations}

The Quantum $A_{\infty}$ relations are defined via the construction of a natural morphism of modular operads 
$\Omega:\mc{F}_D\gamma\St\longrightarrow \mc{E}_V$.

\subsubsection{Motivation and General Structure of $\Omega:\mc{F}_D\gamma\St\longrightarrow \mc{E}_V$}

The stable $\mb{S}$-module underlying $\St$ is a sequence of graded spaces $\{\St((n,b))\}$ satisfying the stability condition \ref{StableSModDef}, where each complex $\St((n,b))$ is generated in part by permutations $\s\in\mb{S}_n$.  To each permutation $\s$, there corresponds the collection of all possible permutations $\{\tilde{\s}_{ij}\}$ such that $\mu^{\St}_{H_i}(\tilde{\s}_{ij})=\s$, where $H_i\longrightarrow *_{n,b}^l$ is a contraction of stable graphs for each $i$. 

 The Feynman differential $d_{\mc{F}}$, by employing the adjoint maps $(\mu_{H_i}^{\St})^*$, extracts $\{\tilde{\s}_{ij}\}$ from $\s$, and the morphism from the Feynman Transform of $\St$ to the modular operad $\mc{E}_V$ provides the structure needed to define relations between vectors $\{v_k\}\subset\mc{E}_V((n,b))$.  
 
 When this structure is applied to moduli spaces of Riemann surfaces with boundary lying on Lagrangian submanifolds of a given Calabi-Yau manifold, the Feynman differential counts all moduli spaces of Riemann surfaces which degenerate to the Riemann surface corresponding to $\s$.  
 
 The above can be expressed concisely as $\Omega\circ d_{\mc{F}}=d_{\mc{E}_V}\circ \Omega$.

\subsubsection{The Twist of $\St$}

We showed in Section \ref{evsings} (reproduced for completeness from Section 3 of \cite{Bar}) that $\mc{E}_V$ is a modular $\mc{K}^{-1}\otimes\bigotimes_{\mm{Edge}(G)}{\bigwedge}^2\mb{C}^{\{s_e,t_e\}}$-operad.  In this section we will show that  $\mc{F}_D\gamma\St$ is a modular operad of the same twist. 

Set $\gamma((n,b)) = sgn_n[-n+b-2]$.  We include $sgn_n$ in order to give the Feynman transform the appropriate orientation, derived in the following lemma.





\begin{lem}
$\gamma\St$ is a modular $\mc{K}^2\mathcal{T}$ where $\mathcal{T}$ is defined by $\mathcal{T}(G) = \bigotimes_{\mm{Edge}(G)}\bigwedge^2\mb{C}^{\{s_e,t_e\}}$.
\end{lem}
\begin{proof}
 $\St$ is a modular $\mc{K}^{-2}\mc{B}$-operad, so $\gamma\St$ is a modular $D_{\gamma}\mc{K}^{-2}\mc{B}$-operad.  We must therefore show $D_{\gamma}\mc{K}^{-2}\mc{B}=\mc{K}^2\mathcal{T}$.    
 
Let $E=|\mm{Edge}(G)|$, $V=|\mm{Vert}(G)|$, and $n=|\mm{Leg}(G)|$.
For any trivalent graph $G$ with $b(v)=0$ for all $v\in\mm{Vert}(G)$, we have
 \begin{eqnarray*}
   \sum_{\mm{Vert}(G)}n(v) &=& n+2E \\
    &=& n+2(b_1-1+V), 
 \end{eqnarray*}
where the second equality follows from Proposition \ref{EdgeVsVertices}.  We can identify $sgn_n[n]$ with $\mm{Det}(\mm{Leg}(v))$, where $\mm{Det}(S):=\mm{Det}(\mb{C}^S)=\bigwedge^{|S|}(\mb{C}^S)[|S|]$ as in Section 2.1 of \cite{Bar}.  Indeed, the representation 
$$\mb{S}_n\times {\bigwedge}^{|S|}(\mb{C}^S)\longrightarrow {\bigwedge}^{|S|}(\mb{C}^S)$$
$$(\s,\ell_1\wedge\cdots\wedge \ell_{|S|})\mapsto \ell_{\s(1)}\wedge\cdots\wedge \ell_{\s(|S|)}$$
is given by checking the parity of $\s$ for all $\s\in\mb{S}_n$.  The following equalities therefore hold at the level of $\mb{S}_n$-representations.  Recall Definition \ref{CocycleExample}, and write $b$ for $b_1+\sum_{\mm{Vert}(G)} b(v)$.  Then
 \begin{eqnarray}\label{taococycle}
   D_{\gamma}\mc{K}^{-2}\mc{B}(G) &=& D_{\gamma}(G)\mc{K}^{-2}(G)\mc{B}(G) \nn\\
    &=& \gamma((n,b))\otimes\bigotimes\gamma^{-1}((n(v),b(v)))\otimes\mb{C}[-2E]\otimes\mb{C}[b_1] \nn\\
    &=& \mb{C}[-2E+b_1]\otimes sgn_n[-n+b-2]\otimes\bigotimes_{\mm{Vert}(G)}(sgn_{n(v)}[-n(v)+b(v)-2])^* \nn\\
    &=& \mb{C}[-2E+b_1-n+b-2+\sum n(v)-\sum b(v)+2V]\otimes{\bigwedge}^n\mb{C}^{\mm{Leg}(G)}  \nn\\
    && \otimes\,\,(\bigotimes_{\mm{Vert}(G)}{\bigwedge}^{n(v)}
    \mb{C}^{\mm{Leg}(v)})^* \nn\\
    &=& \mb{C}[2E]\otimes{\bigwedge}^n\mb{C}^{\mm{Leg}(G)}\otimes({\bigwedge}^{\sum n(v)}\mb{C}^{\mm{Flag}(G)})^* \nn\\
    &=& \mc{K}^2(G)\otimes{\bigwedge}^n\mb{C}^{\mm{Leg}(G)}\otimes({\bigwedge}^{2E+n}\mb{C}^{\mm{Flag}(G)})^* \nn\\
    &=& \mc{K}^2(G)\otimes{\bigwedge}^n\mb{C}^{\mm{Leg}(G)}\otimes({\bigwedge}^n\mb{C}^{\mm{Leg}(G)})^*\otimes
    ({\bigwedge}^{2E}\mb{C}^{\mm{Flag}(G)\setminus\mm{Leg}(G)}) \nn\\
    &=& \mc{K}^2(G)\otimes({\bigwedge}^{2E}\mb{C}^{\mm{Flag}(G)\setminus\mm{Leg}(G)}) \nn\\
    &=& \mc{K}^2(G)\otimes\bigotimes_{\mm{Edge}(G)}{\bigwedge}^2\mb{C}^{\{s_e,t_e\}} \nn \\
    &=& \mc{K}^2(G)\otimes \mathcal{T}(G)
 \end{eqnarray}
 \end{proof}

The operads $\mc{F}_D\gamma\St$ and $\mc{E}_V$ are therefore modular
\begin{eqnarray}
  (D_{\gamma}\mc{K}^{-2}\mc{B})^{\vee}  &=& \mc{K}(D_{\gamma}\mc{K}^{-2}\mc{B})^{-1} \nn \\
   &=& \mc{K}(\mc{K}^2 \mathcal{T})^{-1} \nn\\
   &=& \mc{K}^{-1}\mathcal{T}
\end{eqnarray}
operads.

\subsubsection{Element-Level Structure of $\mc{F}_D\gamma\St\longrightarrow \mc{E}_V$}\label{ev}

As shown in Section 4 of \cite{Bar}, the algebra structure 
\begin{equation}\label{morphism}
    \Omega:\mc{F}_D\gamma\tilde{\mb{S}}[t]\longrightarrow \mc{E}_V
\end{equation}
is defined by the collection of stable $\mb{S}$-module morphisms 
\begin{equation}\label{mhat}
    \Omega_{n,b}:(\gamma\St((n,b)))^*\longrightarrow \mc{E}_V((n,b))
\end{equation}
\begin{equation*}
    [(\s_1\,\cdots\,\,\s_{b-2g+1})t^g[-n+b-2]]^*\mapsto \sum_{i_1\cdots \,i_n} \theta_{i_1\cdots i_n} v_{i_1}\otimes\cdots\otimes v_{i_n}
\end{equation*}
which respect the differentials $d_{\mc{F}}$ and $d_{\mc{E}_V}$.  We write the element of $\mc{E}_V((n,b))$ represented by $\theta=(\theta_{i_1...i_n})$ in the basis $\{v_i\}$ as $\overline\mu_{n,b}^{\theta}$.  

More explicitly, the 
morphism \eqref{morphism}
is defined by the collection of morphisms \eqref{mhat} along with a collection of commutative diagrams of the form
\begin{equation}\label{FeynmanSquare}
    \begin{diagram}
\node{\mc{F}_D\gamma\tilde{\mb{S}}[t]((n,b))}\arrow{s,l}{d_{\mc{F}}}\arrow{e,t}{\Omega_{n,b}}\node{\mc{E}_V((n,b))}\arrow{s,r}{d_{\mc{E}_V}}\\
\node{\mc{F}_D\gamma\tilde{\mb{S}}[t]((n,b))}\arrow{e,b}{\Omega_{n,b}}\node{\mc{E}_V((n,b))}
\end{diagram}
\end{equation}
indexed by pairs $(n,b)$ satisfying the stability condition $2b+n-2>0$, as defined for stable graphs in Definition \ref{stablegraph}. 


In general, the differential $d_{\mc{F}}$ of the Feynman Transform of a modular operad $P$ has the form $d_{P^*}+\sum_{H}d_{(\mu^P_{H\rightarrow G})^*}$, where the sum is taken over all $H$ such that $H/e\simeq G$ for some edge $e$ of $H$.  Because $\gamma\St((n,b))$ is concentrated in a single degree for all pairs $(n,b)$, the differential $d_{\gamma\St^*}$ can be ignored in our case.  

Restricting to $\gamma\tilde{\mb{S}}[t]((n,b)))^*$, the chain-complex diagram involving $d_{\mc{F}}$ and $d_{\mc{E}_V}$ has the form
\begin{equation}\label{diffdiagram}
    \begin{diagram}
\node{(\gamma\tilde{\mb{S}}[t]((n,b)))^*}\arrow{e,t}{\Omega_{n,b}}\arrow{s,l}{d_{\mc{F}}}\node{\mc{E}_V((n,b))}\arrow{s,r}{d_{\mc{E}_V}}\\
\node{\bigoplus_{H/e\simeq *_{n,b}}(D^{\vee}(H)\otimes \gamma\tilde{\mb{S}}[t]((H))^*)_{\mm{Aut}(H)}}\arrow{e,t}{\Omega_{n,b}}\node{\mc{E}_V((n,b))}
\end{diagram}
\end{equation}
The morphism $\Omega$, restricted to the lower left node of Diagram \ref{diffdiagram}, decomposes as
\begin{equation}\label{compdiagram}
    \begin{diagram}
\node{D^{\vee}(H)\otimes \gamma\tilde{\mb{S}}[t]((H))^*}\arrow{e,t}{\Omega}\arrow{s,l}{\pi}\node{D^{\vee}(H)\otimes\mc{E}_V((H))}\arrow{s,r}{\mu_H^{\mc{E}_V}}\\
\node{(D^{\vee}(H)\otimes \gamma\tilde{\mb{S}}[t]((H))^*)_{\mm{Aut}(H)}}\arrow{e,t}{\Omega}\node{\mc{E}_V((n,b))}
\end{diagram}
\end{equation}
where $\pi$ is again projection onto coinvariants.

The element-level combination of Diagrams \ref{diffdiagram} and \ref{compdiagram} takes the form
\begin{equation}\label{newelementleve}
\begin{diagram}
\node{((\s_1\,\cdots\,\,\s_{b-2g+1})t^g)^*}\arrow{s,l}{\pi\circ\, \sum_H d_{\left(\mu^{\gamma\St}_{H}\right)^*}}\arrow[1]{e,t}{\Omega}
\node[1]{\overline{\mu}_{n,b}^{\,\theta}}\arrow{s,r}{d_{\mc{E}_V}}\\
\node{\begin{tabular}{cc}
 $\sum_e(\alpha\otimes\s_e' t^{g'}\otimes\s_e'' t^{g''})^*$  \\
  \qquad + $\sum_{\ell}(\beta\otimes\s_{\ell}'''t^{g-1})^*$ \\
\end{tabular}}\arrow[1]{e,t}{\mu_H^{\mc{E}_V}\circ\,\Omega}\arrow{s,l}{\bigoplus_{H}\Omega|_{H}}\node[1]{
\begin{tabular}{cc}
$\sum_e(-1)^{\epsilon_i}\overline{\mu}_{k,b'}^{\,\vartheta'}\circ_i \overline{\mu}_{n-k+2,b''}^{\,\vartheta''}$\\
\qquad + $\sum_{\ell}(-1)^{\epsilon_j}\overline{\mu}_{n,b}^{\,\vartheta'''}$
\end{tabular}}\\
\node{\begin{tabular}{cc}
$\sum_e\overline{\mu}_{k,b'}^{\,\theta'}\otimes \overline{\mu}_{n-k+2,b''}^{\,\theta''}$\\
\qquad +$ \sum_{\ell}\overline{\mu}_{n+2,b-1}^{\,\theta'''}$
\end{tabular}}\arrow{ne,b}{\mu_H^{\mc{E}_V}}
\end{diagram}
\end{equation}
where the first sum in each of the lower nodes is indexed by the set of insertions of non self-intersecting edges, and the second sum in each node by insertions of loops.  The terms $\epsilon_i$, as defined in \eqref{signsEV} and \eqref{signsEV2}, yield the appropriate signs.  We decorate the image vectors with $\vartheta$ to indicate possibly new tensor coefficients arising from contraction via $B$.  Note that 
$$\pi_{ff'}(\s_e'\s_e''(ff')),\pi_{ff'}(\s'''_{\ell}(ff'))=\s_1\,\cdots\,\,\s_{b-2g+1}$$  
We also have that if $v\in (D^{\vee}(H)\otimes \gamma\tilde{\mb{S}}[t]((H))^*)_{\mm{Aut}(H)}$ in Diagram \ref{diffdiagram}, then $v=\pi(\tilde{v})$ for some lift $\tilde{v}$ of $v$, and $\Omega(v)=(\Omega\circ \pi)(\tilde{v})=(\mu_H^{\mc{E}_V}\circ\Omega)(\tilde{v})$.


The equation $d_{\mc{E}_V}\circ\Omega=\Omega\circ d_{\gamma\St^*}$ has the expanded form
\begin{eqnarray}\label{chaincomplexequation}
   && d_{\mc{E}_V}(\hat{m}_{n,b}(\s[-n+b-2])^*) \nn\\
   &=& (\mu_{G_{n,b-1}^1}^{\mc{E}_V}\circ\hat{m}_{n+2,b-1}\circ(\mu_{G^1_{n,b-1}}^{\gamma\St})^*)(\s[-n+b-2]^*)) \nn\\
   &+& \sum_{\substack{H\in\Gamma((n,b))\\H/e\simeq G_{n,b}}}(\mu_H^{\mc{E}_V}\circ(\hat{m}_{k,b_1}\otimes\hat{m}_{n-k+2,b_2})\circ(\mu_H^{\gamma\St})^*)(\s[-n+b-2]^*))
\end{eqnarray}

\subsection{The Quantum $A_{\infty}$-Relations as a Generalization of the $A_{\infty}$-Relations}

\begin{defn} A Quantum $A_{\infty}$-algebra is a collection of maps
\begin{eqnarray}\label{Qdef}
\hat{m}_{n,b}:\gamma\St((n,b))\longrightarrow \mc{E}_V \\
(\s[-n+b-2])^*\mapsto \overline{\mu}_{n,b}^{\,\theta} \nn
\end{eqnarray}
satisfying Equation \eqref{chaincomplexequation}.  As before, $V$ is any chain complex such that its homogeneous subspaces are finite-dimensional, and is endowed with a non-degenerate bilinear form $B$ such that $B(dx,y)+(-1)^{|x|}B(x,dy)=0$.  
\end{defn}

\begin{defn}
An (non-unital) $A_{\infty}$-algebra $A$ is a cohomologically $\mb{Z}$-graded $k$-vector space $$A=\bigoplus_{p\in\mb{Z}}A^p$$with
graded $k$-linear maps, for $d\geq1$,$$m_d:A^{\otimes d}\longrightarrow A$$of degree $2-d$ satisfying for each
$d\geq1$ the relation
\begin{equation}\label{A-inf}
    \sum_{\substack{1\leq p\leq d\\0\leq q\leq d-p}}  (-1)^{\deg a_1+\cdots+\deg a_q-q}m_{d-p+1}(a_d,\dots,a_{p+q+1}
    ,m_p(a_{p+q},\dots,a_{q+1}),a_q,\dots,a_1)=0.
\end{equation}
The sign $(-1)^{\deg a_1+\cdots+\deg a_q-q}$ is given by Seidel in \cite{Seidel}.

Given a bilinear form $B:A\otimes A\longrightarrow \mb{C}$ we construct an isomorphism $A^*\longrightarrow A$ via the identification 
$\mm{Hom}^{|B|}(A\otimes A,\mb{C}) = \mm{Hom}^0((A\otimes A)[-|B|],\mb{C})$, and use this to construct an isomorphism 
$$\mm{Hom}(A^{\otimes d},A) \longrightarrow (A^{\otimes d})^*\otimes A\longrightarrow A^{\otimes d}\otimes A,$$
where the first map is the natural isomorphism.  The first isomorphism preserves degree and the second adds $d|B|$.     

We can therefore express \eqref{A-inf} in its tensor
form:
\begin{equation}\label{TensorA-inf}
\sum_{\substack{1\leq p\leq d\\0\leq q\leq d-p}}(-1)^{\deg a_1+\cdots+\deg a_q} \,\overline{\mu}_{d-p+1}\circ_{q+1}\overline{\mu}_p=0,
\end{equation}
where $\overline{\mu}_i := \varphi(m_i)$, $\circ_i$ is contraction via $B$, and the tensors $\overline{\mu}_d$
sit in degree $2+d(|B|-1)$.
\end{defn}

\begin{theorem}\label{A_infThm}
The Quantum $A_{\infty}$-relations reduce to the usual $A_{\infty}$ relations for $b=0$.  
\end{theorem}

\begin{proof}
When $b=0$ Equation \ref{chaincomplexequation} takes the form
\begin{eqnarray}\label{reducedchaincomplexequation}
  0 &=& \sum_{\substack{H\in\Gamma((n,b))\\H/e\simeq G_{n,b}}}(\mu_H^{\mc{E}_V}\circ(\hat{m}_{k,b_1}\otimes\hat{m}_{n-k+2,b_2})\circ(\mu_H^{\gamma\St})^*)(\s[-n+b-2]^*)) \nn\\
  &=& \sum_{\substack{H\in\Gamma((n,b))\\H/e\simeq G_{n,b}}} (-1)^{\epsilon_{\beta}}\overline{\mu}_{n-k+2,0}\circ_{\beta}\overline{\mu}_{k,0}
\end{eqnarray}
We need only show the signs $(-1)^{\epsilon_{\beta}}$ are of the form given in Equation \ref{TensorA-inf}.  To this end, we prove the following lemma.
\end{proof}

\begin{lem}
Let $\overline{\mu}_n = p_1\otimes p_2\otimes\cdots \otimes p_n$ and $\overline{\mu}_m = q_1\otimes\cdots\otimes q_{\beta}\otimes\cdots\otimes q_m$, where we adopt the convention of ordering the points $p_i$ counter-clockwise as described in the discussion following Definition \ref{fatgraphorientation}.  The contraction of the element $\overline{\mu}_n\otimes_{\beta} \overline{\mu}_m$ via $B$ for $B$ odd yields the sign
\begin{equation}\label{ainfsign}
  (-1)^{\sum\limits_{i=1}^{\beta-1}|q_i|}
\end{equation}

\end{lem}

\begin{proof}
Recall Equation \eqref{signsEV}.  The contraction via $B$ of $\overline{\mu}_n\otimes_{\beta} \overline{\mu}_m$  yields the sign $(-1)^s$ where

\begin{equation}\label{Abouzaidsign}
s = \sum\limits_{i=1}^{\beta-1}|q_i|(|\overline{\mu}_m|-|q_i|)
+|B|(|\overline{\mu}_n|-|p_n|)+|p_n|+\sum\limits_{j=1}^{\beta-1}|q_j|(|\overline{\mu}_n|+|\overline{\mu}_m|-|p_n|-|q_{\beta}|-|q_j|)
\end{equation}
and is taken modulo 2.

We have that $|\overline{\mu}_{d,b}|=2+d(|B|-1)$ for any $d,b$ satisfying the stability constraint.  Substituting $2+n(|B|-1)$ for $|\overline{\mu}_{n,b}|$ and $2+m(|B|-1)$ for $|\overline{\mu}_{m,b}|$ in \eqref{Abouzaidsign} yields   

\begin{eqnarray*}
s&=& \sum\limits_{i=1}^{\beta-1} |q_i|(2+m(|B|-1)-|q_i|) + |B|(2+n(|B|-1)-|p_n|)+|p_n|  \\
&+& \sum\limits_{j=1}^{\beta-1} |q_j|(2+n(|B|-1) + 2+m(|B|-1) -|p_n|-|q_{\beta}|-|q_j|) \mod 2
\end{eqnarray*}

Because $B$ is odd by assumption, we make the substitution $|B|=1\mod 2$, which gives 
\begin{eqnarray*}
s&=&  \sum\limits_{i=1}^{\beta-1} |q_i|(-|q_i|) + |B|(-|p_n|)+|p_n| +  \sum\limits_{j=1}^{\beta-1} |q_j|(-|p_n|-|q_{\beta}|-|q_j|)\mod 2 \\
&=&  \sum\limits_{i=1}^{\beta-1} |q_i|(-|q_i|) + (-|B|+1)|p_n| +  \sum\limits_{j=1}^{\beta-1} |q_j|(-|B|-|q_j|)\mod 2\\
&=&  \sum\limits_{i=1}^{\beta-1} |q_i| +  \sum\limits_{j=1}^{\beta-1} |q_j|(-|B|-|q_j|)\mod 2 \\
&=&  \sum\limits_{i=1}^{\beta-1} |q_i| +  \sum\limits_{j=1}^{\beta-1} |q_j|(1-|q_j|)\mod 2 \\ 
&=&  \sum\limits_{i=1}^{\beta-1} |q_i|  \mod 2,
\end{eqnarray*}
where the second equality follows from the implicit assumption that $B(p_n,q_{\beta})\neq0$ and thus $|B|=|p_n|+|q_{\beta}|$.

\end{proof}




\section{The Quantum $A_{\infty}$-Structure on the Elliptic Curve}
We prove the Quantum $A_{\infty}$-relations on the Elliptic curve 
are obtained by counting degenerations of one-dimensional moduli spaces of tropical Morse graphs of genus zero or one.

\subsection{The Modular Operad $\mc{E}_{\mc{L}}$}\label{ModOPSection}
In order to extend the Quantum $A_{\infty}$-algebra structure from chain complexes to the Fukaya category of a Calabi-Yau manifold $\mc{X}$, we construct an operad $\mc{E}_{\mc{L}}$, modeled after $\mc{E}_V$, from the Hom-spaces of the Lagrangian submanifolds of $\mc{X}$.  Moduli spaces of Riemann surfaces with boundary components lying on Lagrangian submanifolds and the relations between them can then be parameterized by stable graphs and their contractions.

Let $b=2g+i_{\s}-1\in\mb{Z}_{\geq0}$ and let $\{L_{i1},\dots,L_{id_i}\}_{i=1}^{b-2g+1}$ be a finite sequence of cyclic chains of Lagrangians, where $g,i_{\s}\in\mb{Z}$ are such that $g\geq0$ and $i_{\s}\geq1$.  Define for any $\s\in S_n$ the element
\begin{equation}\label{EL}
    \overline{\mu}_{n,b,\s}\in\mc{E}_{\mc{L}}((n,b))=\bigotimes_{i=1}^{b-2g+1}\bigotimes_{j=2}^{d_i}
    \mm{Hom}(L_{ij},L_{i(j-1)})\otimes\mm{Hom}(L_{i1},L_{id_i})
\end{equation}
by
\begin{equation}\label{munb}
\overline{\mu}_{n,b,\s}=\sum_{\substack{\mm{tensors}\\p_1\otimes\dots\otimes p_{n}}}
\sum_{G_{n,b,\s}^{\mm{trop}}}(-1)^{s(\phi)}q^{\deg\phi}
p_1\otimes\cdots\otimes p_n,
\end{equation}
where $\phi\in G_{n,b,\s}^{\mm{trop}}$, $\sum_{i=1}^{b-2g+1} d_i=n$, and the points $p_1,\dots,p_n$ are partitioned into $b-2g+1$ cycles $\s_1\cdots\,\s_{b-2g+1}=\s$ according to the cyclic chains $\{L_{i1},\dots,L_{id_i}\}_{i=1}^{b-2g+1}$. 

The tensors $\overline{\mu}_{n,b,\s}$ of $\mc{E}_{\mc{L}}$ are defined by summing over tropical Morse graphs with $n$ legs, genus $b$, and moduli dimension zero.  The dimension of $G_{n,b,\s}^{\mm{trop}}$ is given by $n-2+2b-\sum\deg p_i$, so $\overline{\mu}_{n,b,\s}$ sits in cohomological degree
\begin{eqnarray*}
    n-\deg(p_1\otimes\cdots\otimes p_n) &=& n-\sum\deg p_i \\
     &=& n-(n-2+2b) \\
    &=& 2-2b
  \end{eqnarray*}

 The sign $s(\phi)$ and degree $\deg\phi$ of a Tropical Morse Graph $\phi$ were defined in Section \ref{ConstructionOfRS}.  The compositions are defined as in Equations \eqref{signsEV} and \eqref{signsEV2}, and because $\mc{E}_{\mc{L}}((n,b))$ is concentrated in a single degree for any $(n,b)$, we have $d_{\mc{E}_{\mc{L}}}=0$.

Notice that if $b=0$, then $\overline{\mu}_{n,0}$ corresponds to a map
\begin{equation}\label{usualmns}
    m_{n-1}:\mm{Hom}(L_1,L_2)\otimes\cdots\otimes\mm{Hom}(L_{n-1},L_n)\longrightarrow \mm{Hom}(L_1,L_n)
\end{equation}
via the isomorphism
\begin{equation}\label{tensiso}
    \bigotimes_{i=2}^{n-1}\mm{Hom}(L_i,L_{i-1})\otimes\mm{Hom}(L_1,L_n)\longrightarrow \mm{Hom}(\bigotimes_{i=2}^{n-1}\mm{Hom}(L_{i-1},L_i),\mm{Hom}(L_1,L_n))
\end{equation}
\begin{equation*}
    \bigotimes_{i=1}^n p_i\mapsto (\bigotimes_{i=1}^{n-1} q_i\mapsto \prod_{i=1}^{n-1}B(p_i,q_i)p_n)
\end{equation*}
where the degree 1 bilinear form $B:\mm{Hom}(L_i,L_j)\otimes\mm{Hom}(L_j,L_i)\longrightarrow \mb{C}$
is defined by \begin{equation*}
   (-1)^{|p_i|}B(p_i,q_i) =
    \begin{cases}
    1 &\text{if $q_i=p_i^*$}\\
    0 &\text{otherwise}
    \end{cases}
    \end{equation*}
for $1\leq i\leq n-1$.

The bilinear form $B$ is anti-symmetric, so $\mc{E}_{\mc{L}}$, like $\mc{E}_V$, is a modular $$D = \mc{K}^{-1}\otimes\bigotimes_{\mm{Edge}(G)}{\bigwedge}^2\mb{C}^{\{s_e,t_e\}}$$ operad.  See Section \ref{evsings} for details.



\subsection{The Existence of the Quantum $A_{\infty}$-Relations on the Elliptic Curve}

Both the $A_{\infty}$-relations and the Quantum $A_{\infty}$-relations are obtained by considering degenerations of 1-dimensional moduli spaces of tropical Morse graphs.  In the case of the elliptic curve we assume $\deg(p)\in\{0,1\}$, so $1=\dim G_{n,b,\s}^{\mm{trop}} = n-2+2b-\sum\deg p_i$ only if $b=0$ or $b=1$.  As was shown in Theorem \ref{A_infThm}, the $b=0$ case corresponds to the usual $A_{\infty}$-relations, so proving the existence of the Quantum $A_{\infty}$-relations on the elliptic curve reduces to the $b=1$ case of Equation \eqref{chaincomplexequation}.




\begin{theorem}\label{MainTheorem1} (Main Theorem)
There exists a non-trivial Quantum $A_{\infty}$ structure on the elliptic curve, i.e., there exists a non-trivial algebra structure 

\begin{equation}\label{morphismCategorical}
    \Omega:\mc{F}_D\gamma\tilde{\mb{S}}[t]\longrightarrow \mc{E}_{\mc{L}},
\end{equation}
which respect the differentials $d_{\mc{F}}$ and $d_{\mc{E}_{\mc{L}}}$. 

\end{theorem}




\begin{proof}

Because $d_{\mc{E}_{\mc{L}}}=0$, we must show that the elements $\overline{\mu}_{n,1,\s}\in \mc{E}_{\mc{L}}((n,1))$ satisfy the the following simplified form of the equation \eqref{chaincomplexequation}:
\begin{equation}
(-1)^{\epsilon}\overline{\mu}_{n,1,\s}+\sum_e(-1)^{\epsilon_i}\overline{\mu}_{k,b_{v_1},\rho}\circ_i\overline{\mu}_{n-k+2,b_{v_2},\upsilon}=0
\end{equation}
Here, $\pi_{ff'}(\rho\upsilon(ff'))=\s$, $b_{v_1}+b_{v_2}=1$, the sum is taken over insertions of non self-intersecting edges $e$, and the signs 
$(-1)^{\epsilon_i},(-1)^{\epsilon}$ are given by Equations \eqref{signsEV} and \eqref{signsEV2}, respectively.  Note that there is exactly one term corresponding to the insertion of a self-intersecting edge, as the equation $\pi_{ff'}(\tilde{\s}(ff'))=\s$ has a unique cycle solution $\tilde{\s}$ for $b=1$ .

Let $H$ be any stable graph which contracts to a single-vertex graph via the contraction of exactly one edge.  Recall that a $b=1$ vertex is necessarily decorated with two cycles, so that when the adjoint map 
$(\mu_H^{\gamma\St})^*$ is applied to a vector represented by a single-vertex graph, the result is either
a single-vertex graph with a self-intersecting $b=0$ edge, or a two-vertex tree with exactly one $b=0$ vertex and one $b=1$ vertex.  In the latter case, one of the cycles decorating the $b=1$ vertex is comprised solely of the unique flag which both comprises the edge and is attached to the $b=1$ vertex.  The flags comprising the edge contracted via $\mu_H$ will be denoted as before by $f$ and $f'$.  For each topologically unique way of inserting a non self-intersecting edge, there are two labelings of the vertices $\{v_1,v_2\}=\mm{Vert}(H)$.

The compositions $\overline{\mu}_{k,b_{v_1},\rho}\circ_i\overline{\mu}_{n-k+2,b_{v_2},\upsilon}$ do not exist.  Indeed, these compositions are defined by degenerations of moduli spaces of tropical Morse graphs such that the unique free vertex lies away from the unique cycle, each corresponding topologically to a disk bubbling off from an annulus.  By Lemma \ref{cyclelemma}, these moduli spaces are one dimensional only if the single free vertex of a point on the moduli space lies on the unique cycle.  This leaves us with $(-1)^{\epsilon}\overline{\mu}_{n,1,\s}=0$, which reduces to $\overline{\mu}_{n,1,\s}=0$.

By Lemma \ref{main} one need only show that for each degenerate pair $u$ and $v$, the contractions via the antisymmetric form $B$ of the tensors corresponding to $u$ and $v$ yield opposite signs.

Let $p,q$ be the matching pair of corners of one degeneration and $p',q'$ be the matching pair for the other degeneration.  The two tensors of the degenerations in question then have the form $p\otimes p_1\otimes\cdots\otimes p_m\otimes q$ and $p'\otimes q_1\otimes\cdots\otimes q_m\otimes q'$, where $\sum_{i=1}^m\deg p_i=\sum_{j=1}^m q_j$ by lemma \ref{newdeg1}.  Finally, since $\deg p+\deg p'=\deg q+\deg q'=1$, the antisymmetry of $B$ gives $$B(p\otimes p_1\otimes\cdots\otimes p_m\otimes q)=-B(p'\otimes q_1\otimes\cdots\otimes q_m\otimes q'),$$ and this completes the proof.

\end{proof}

\section{Conclusion}






With an explicit formulation of the Quantum $A_{\infty}$-relations, one may now use the partition function of Barannikov and Kontsevich to compute cohomology classes in $\overline{\mc{M}}_{n,1}$, and the mirror map in \cite{Clay} may now be extended to include the genus 1 elements $\overline{\mu}_{n,1,\s}$.

\end{document}